\documentclass{amsart}

\usepackage[applemac]{inputenc}
\usepackage[T1]{fontenc}
\usepackage{amsfonts}
\usepackage{amsmath}
\usepackage{amssymb}
\usepackage{amsthm}
\usepackage{a4wide}
\usepackage{graphicx}
\usepackage{mathrsfs}
\usepackage{hyperref} 

\usepackage{enumitem}

\usepackage{xcolor}

\usepackage{marginnote}

\makeatletter 
\def\@tocline#1#2#3#4#5#6#7{\relax
  \ifnum #1>\c@tocdepth 
  \else
    \par \addpenalty\@secpenalty\addvspace{#2}%
    \begingroup \hyphenpenalty\@M
    \@ifempty{#4}{%
      \@tempdima\csname r@tocindent\number#1\endcsname\relax
    }{%
      \@tempdima#4\relax
    }%
    \parindent\z@ \leftskip#3\relax \advance\leftskip\@tempdima\relax
    \rightskip\@pnumwidth plus4em \parfillskip-\@pnumwidth
    #5\leavevmode\hskip-\@tempdima
      \ifcase #1
       \or\or \hskip 1em \or \hskip 2em \else \hskip 3em \fi%
      #6\nobreak\relax
    \dotfill\hbox to\@pnumwidth{\@tocpagenum{#7}}\par
    \nobreak
    \endgroup
  \fi}
\makeatother

\DeclareMathAlphabet{\mathpzc}{OT1}{pzc}{m}{it}
\newcommand{\hh}{\mathbb{H}}

\newtheorem{theorem}{Theorem}[section]
\newtheorem*{theorem*}{Theorem}
\newtheorem{proposition}[theorem]{Proposition}
\newtheorem{corollary}[theorem]{Corollary}
\newtheorem{lemma}[theorem]{Lemma}
\theoremstyle{definition}
\newtheorem{definition}[theorem]{Definition}
\newtheorem{example}[theorem]{Example}
\newtheorem{remark}[theorem]{Remark}

\title{Log-biharmonicity and a Jensen formula in the space of quaternions} 

\author[Altavilla]{Amedeo Altavilla}
\address{Altavilla Amedeo: Dipartimento Di Matematica, Universit\`a di Roma Tor Vergata, Via Della Ricerca Scientifica 1, 00133, Roma, Italy}
\email{altavilla@mat.uniroma2.it}
\author[Bisi]{Cinzia Bisi}
\address{Bisi Cinzia: Dipartimento di Matematica e Informatica, Universit\`a di Ferrara, via Machiavelli 35, I-44121 Ferrara, Italy}
\email{bsicnz@unife.it}

\begin{document}


\begin{abstract}
Given a complex meromorphic function, it is well defined its Riesz measure
in terms of the laplacian of the logarithm of its modulus. 
Moreover, related
to this tool, it is possible to prove the celebrated Jensen formula.
In the present paper, using among the other things the fundamental solution for the bilaplacian,
we introduce a possible generalization of these two concepts in the space of quaternions, obtaining new interesting Riesz measures and global (i.e. four dimensional), Jensen formulas. 
\end{abstract}
\maketitle

\tableofcontents

\section*{Introduction}
In classical complex analysis, \textit{harmonic functions} are defined to be the solutions
of the \textit{Laplace equation} $\Delta f=0$ and, as it is well known, they are characterized by satisfying the mean value property: 
$f:\Lambda\subset\mathbb{C}\rightarrow\mathbb{C}$ is harmonic if and only if for any $z_{0}\in\Lambda$ such that the disc $D(z_{0},\rho)$ centered in $z_{0}$ with radius $\rho$ is contained in $\Lambda$, it holds:
\begin{equation*}
f(z_0)=\frac{1}{2 \pi} \int_{0}^{2 \pi} f(z_0 + \rho e^{i \theta}) d \theta.
\end{equation*}

Since the logarithm of the modulus of any analytic function $f$ is a harmonic
function outside the zero set of $f$, then it is possible to prove firstly that
$\log|z|$ is a multiple of the \textit{fundamental solution} of the Laplace equation and, moreover, the so celebrated \textit{Jensen formula}.

Entering into the details, it is well known that
if $u$ is a subharmonic function on a domain $D \subset \mathbb{C},$ with $u \not\equiv - \infty,$ then the {\it generalized laplacian} of $u$ is the \textit{Radon measure} $\Delta u$ on $D,$ i.e. the laplacian in the sense of distributions. 
The potential $p_{\mu}$ associated to a measure $\mu$  can be seen as the distributional convolution of $\mu$ with the locally integrable function $\log |z|$.
Then, we can state the following theorem which asserts that $\Delta p_\mu$ is the convolution of $\mu$ with a $\delta-$function, i.e. a multiple of $\mu$ itself.\begin{theorem}
Let $\mu$ be a finite Borel measure on $\mathbb{C}$ with compact support. Then:
$$
\Delta p_\mu =\mathfrak{c} \mu,
$$
where $\mathfrak{c}$ is a constant depending on the convention used to compute the laplacian.
\end{theorem}

A particular case of the previous theorem is the following:
\begin{theorem} \label{CxRiesz}
Let $f:\Lambda\subset\mathbb{C}\rightarrow\mathbb{C}$ be a holomorphic function, with $f \not\equiv 0.$ 
Then $\Delta \log |f|$ is composed of $ \mathfrak{c}-$Dirac deltas on the zeros of $f,$ counted with their multiplicities.
\end{theorem}
With the last result we may say, in some sense, that the theory of potentials pays back its debt to complex analysis. 
In this paper, an analog of this theorem is obtained in the quaternionic setting, via the use of the bilaplacian over $\mathbb{R}^4,$ instead of the laplacian over $\mathbb{R}^2.$ 

Coming back to Jensen formula, in 1899 Johan Jensen investigated how the mean 
value property for the logarithm of the modulus of a holomorphic function becomes in presence of zeros in the interior of $|z| \le \rho$.
If $f:\Lambda\subset\mathbb{C}\rightarrow\mathbb{C}$ is a holomorphic function such that
$D(0,\rho)\subset\Lambda$,
denoting the zeros of $f|_{D(0,\rho)}$ as $a_1, \cdots , a_n$, taking into account their multiplicities and assuming that $z=0$ is not a zero, he proved that
\begin{equation}\label{jeneq}
\log |f(0)| = \frac{1}{2 \pi} \int_{0}^{2\pi} \log |f(\rho e^{i \theta})| d \theta -\sum\limits_{i=1}^{n} \log \left(\frac{\rho}{|a_i|}\right).
\end{equation}
Nowadays this is called {\it Jensen formula} and it relates the average modulus of an analytic function on a circle 
with the moduli of its zeros inside the circle and, for this reason, it is an important statement in the study of entire functions in complex analysis. 
In this paper we lay the groundwork to generalize also this result over the skew field of the quaternions $\mathbb{H}$. 
To reach this and the previous aim, we will use two particular classes of quaternionic functions
of one quaternionic variable. The first will be the class of slice preserving
regular functions, while the second is a new class
of functions which naturally arises in our theory: PQL functions.

Slice preserving regular functions are regular functions in the sense of  \cite{genstostru, ghiloniperotti}, such that, for any quaternionic imaginary unit
$J\in\mathbb{H}$ (i.e. $J^{2}=-1$), they send the complex line $\mathbb{C}_{J}:={\rm Span}_{\mathbb{R}}\{1,J\}\subset\mathbb{H}$ into itself.

A PQL function $f$ is a function of the following type:
\begin{equation*}
f(x)=a_{0}(x-q_{1})^{M_{1}}a_{1}\dots a_{N-1}(x-q_{N})^{M_{N}}a_{N},
\end{equation*}
where, $\{q_{k}\}_{k=1}^{N}$ and $\{a_{k}\}_{k=0}^{N}$ are finite sets of, possibly repeated, quaternions and $M_{k}=\pm1$ for any $k$.
In our knowledge, this class of functions was never studied  whereas fit very well in the topics 
we are going to introduce.

Even if the intersection between these two families is nonempty, 
PQL functions are not in general regular in the sense of \cite{genstostru, ghiloniperotti}. 
Furthermore, thanks to their
particular expression it is possible to fully describe their zeros and singularities.
Both families will be properly defined and discussed in Section 1. 

We give now a simplified version of the two main theorems of this work.
If $\Omega\subset\mathbb{H}$ is a domain and $f:\Omega\rightarrow\hat{\mathbb{H}}=\mathbb{H} \cup \{ \infty \}$ is any quaternionic function of one quaternionic variable, when it makes sense, we will denote by $\mathcal{Z}(f)$ and $\mathcal{P}(f)$ the
sets of its zeros and ``singularities'', respectively and $\mathcal{ZP}(f)=\mathcal{Z}(f)\cup\mathcal{P}(f)$. 

In the whole paper the open ball centered in zero with radius $\rho$ will be denoted
by $\mathbb{B}_{\rho}$. When $\rho=1$, then we will simply write $\mathbb{B}_{1}=\mathbb{B}$.

The first main theorem of this paper is the quaternionic analogue of Theorem \eqref{CxRiesz} where, instead of the Laplace operator, we use the \textit{bilaplacian} $\Delta^{2}=\Delta\circ\Delta$. To obtain it we firstly reconstruct the fundamental solution for the bilaplacian of $\mathbb{R}^{4}$.

\begin{theorem}[Riesz measure]
Let $\Omega\subseteq\mathbb{H}$ be a domain such that $\overline{\mathbb{B}_{\rho}}\subset\Omega$ for some $\rho>0$. Let $f:\Omega\rightarrow\hat{\mathbb{H}}$ be
a slice preserving regular function or a PQL function. Then
\begin{equation*}
-\frac{1}{48}\Delta^{2}\log|f|\big|_{\mathbb{B}_{\rho}}=\delta_{\mathcal{Z}(f_{|\mathbb{B}_{\rho}})}-\delta_{\mathcal{P}(f_{|\mathbb{B}_{\rho}})},
\end{equation*}
where $\delta_{\mathcal{Z}(f_{|\mathbb{B}_{\rho}})}$ and $\delta_{\mathcal{P}(f_{|\mathbb{B}_{\rho}})}$ are the Dirac measure of the set $\mathcal{Z}(f_{|\mathbb{B}_{\rho}})$ and $\mathcal{P}(f_{|\mathbb{B}_{\rho}})$, respectively.
\end{theorem}

The reader can find all the features of the quaternionic Riesz measure in Section 2 of this paper.
In particular, in Remark~\ref{mix1}, we discuss the case of a product between a slice preserving regular function and a PQL function.

The second main theorem is a quaternionic analogue of the complex Jensen formula \eqref{jeneq}:
\begin{theorem}[Jensen formula]
Let $\Omega\subseteq\mathbb{H}$ be a domain such that $\overline{\mathbb{B}_{\rho}}\subset\Omega$ for some $\rho>0$. Let $f:\Omega\rightarrow\hat{\mathbb{H}}$ be
a slice preserving regular function or a PQL function
such that $f(0)\neq 0,\infty$. 
Then,
\begin{equation*}
\log|f(0)| = \displaystyle\frac{1}{|\partial \mathbb{B}_{\rho}|}\displaystyle\int_{\partial \mathbb{B}_{\rho}}\log|f(y)|d\sigma(y)-\displaystyle\frac{\rho^{2}}{8}\Delta\log|f(x)|_{|x=0}+\Lambda_{\rho}(\mathcal{ZP}(f)),
\end{equation*}
where $\Lambda_{\rho}(\mathcal{ZP}(f))$ is a quaternion depending on zeros and
singularities of $f_{|\mathbb{B}_{\rho}}$. 
\end{theorem}

The details about Jensen formula and its corollaries are illustrated in Section 3. In particular, in Remark~\ref{mix2}, we discuss the case of a product between a slice preserving regular function and a couple of PQL functions.

To state previous theorems and, as a tool for describing what is mentioned in this introduction, in Section 1
we will state the main definitions and results about slice regular functions;
for what concerns this part, we point out that some observations on the structure of 
$\mathcal{ZP}(f)$ for a (semi)regular function are original, even if they were predicted 
by experts in this field (see Corollary \ref{zerocpt}, Lemma \ref{ordpoleslpr} and Corollary \ref{polecpt}). In the same section we will properly introduce PQL functions and the class of $\rho$-\textit{Blaschke} factors,
that are analogues of what was already introduced in \cite{irenebla}. Some properties of this class of functions will be stated. This part is original even if in part inspired by previous works \cite{irenebla, shur}.

Finally,
in the very last subsection, we will list a number of corollaries that follow from
our Jensen formula. 
The first two of them, Corollaries \ref{genjens1} and \ref{genjens2}, 
deal with possible generalizations of the formula, namely, when some zeros or singularities at the boundary of the ball (where
the integral of the formula is computed), occur and, the second one, 
when the function vanishes or it's singular at
the origin. 
After that, in Corollaries \ref{corN} and \ref{corNN}, we give upper bounds on the number of zeros of a
slice regular function under some additional hypotheses.
The following Corollaries \ref{measurePQL} and \ref{measure} give formulas for the computation of some integrals over 
$\mathbb{H}$.\\
The results contained in this paper will be further developed in the understanding
of harmonic analysis on quaternionic manifolds (see \cite{angellabisi}, \cite{bisigentilitori} and \cite{BW}).

\section{Prerequisites about quaternionic functions}

In this section we will overview and collect the main notions and results needed
for our aims. 
First of all, let us denote by $\mathbb	{H}$ the real algebra
of quaternions. An element $x\in\mathbb{H}$ is usually written as 
$x=x_{0}+ix_{1}+jx_{2}+kx_{3}$, where $i^{2}=j^{2}=k^{2}=-1$ and $ijk=-1$.
Given a quaternion $x$ we introduce a conjugation in $\mathbb{H}$ (the usual one),
as $x^{c}=x_{0}-ix_{1}-jx_{2}-kx_{3}$; with this conjugation we define the real
part of $x$ as $Re(x):=(x+x^{c})/2$ and the imaginary part as $Im(x):=(x-x^{c})/2$.
With the just defined conjugation we can write the euclidean square norm of a
quaternion $x$ as $|x|^{2}=xx^{c}$. 
The subalgebra of real numbers will be identified, of course, with the set
$\mathbb{R}:=\{x\in\mathbb{H}\,\mid\,Im(x)=0\}$

Now, if $x$ is such that $Re (x)=0$, then the imaginary part of $x$ is such that 
$(Im(x)/|Im(x)|)^{2}=-1$. More precisely, any imaginary quaternion $I=ix_{1}+jx_{2}+kx_{3}$, such 
that $x_{1}^{2}+x_{2}^{2}+x_{3}^{2}=1$ is an imaginary unit. The set of imaginary units 
is then a $2-$sphere and will be conveniently denoted as follows:
\begin{equation*}
\mathbb{S}:=\{x\in\mathbb{H}\,\mid\, x^{2}=-1\}=\{x\in\mathbb{H}\,\mid\, Re(x)=0, \, |x|=1\}.
\end{equation*}

With the previous notation, any $x\in\mathbb{H}$ can be written as $x=\alpha+I\beta$,
where $\alpha,\beta\in\mathbb{R}$ and $I\in\mathbb{S}$. 
Given any $I\in\mathbb{S}$ we will denote the real subspace of $\mathbb{H}$ 
generated by $1$ and $I$ as:
\begin{equation*}
\mathbb{C}_{I}:=\{x\in\mathbb{H}\,\mid\,x=\alpha+I\beta, \alpha,\beta\in\mathbb{R}\}.
\end{equation*}
Sets of the previous kind will be called \textit{slices}.
All these notations reveal now clearly the \textit{slice} structure of $\mathbb{H}$ as
union of complex lines $\mathbb{C}_{I}$ for $I$ which varies in $\mathbb{S}$, i.e.
\begin{equation*}
\mathbb{H}=\bigcup_{I\in\mathbb{S}}\mathbb{C}_{I},\quad\bigcap_{I\in\mathbb{S}}\mathbb{C}_{I}=\mathbb{R}
\end{equation*}
The following notation will also be useful for some purpose: 
\begin{equation*}
\mathbb{C}_{I}^{+}:=\{x\in\mathbb{H}\,\mid\,x=\alpha+I\beta, \alpha\in\mathbb{R},\beta>0\},\quad I\in\mathbb{S},
\end{equation*}
and sets of this kind will be called \textit{semislices}.
Observe that for any $I\neq J\in\mathbb{S}$, $\mathbb{C}_{I}^{+}\cap\mathbb{C}_{J}^{+}=\emptyset$ and $\mathbb{H}\setminus\mathbb{R}=\cup_{I\in\mathbb{S}}\,\mathbb{C}_{I}^{+}$, we have then the following diffeomorphism $\mathbb{H}\setminus\mathbb{R}\simeq\mathbb{C}^{+}\times\mathbb{S}$.

We denote the $2-$sphere with center
$\alpha\in\mathbb{R}$ and radius $|\beta|$ (passing through $\alpha+I\beta\in\mathbb{H}$), as:
\begin{equation*}
\mathbb{S}_{\alpha+I\beta}:=\{x\in\mathbb{H}\,\mid\,x=\alpha+J\beta, J\in\mathbb{S}\}.
\end{equation*}
Obviously, if $\beta=0$, then $\mathbb{S}_{\alpha}=\{\alpha\}$.


\subsection{Slice functions and regularity}
In this part we will recall the main definitions and features of slice functions. 
The theory of slice functions was introduced in \cite{ghiloniperotti} as a tool to generalize
the one of quaternionic regular functions defined on particular domains introduced in \cite{gentilistruppa, gentilistruppa1}, to more general domains and to all the 
alternative $*-$algebras. 
Even if this more abstract approach seems to be meaningless, it has been proved to be 
very effective in a lot of situations. So, take a deep breath and accept our position
for no more than some pages.

\vspace{0.5cm}
\noindent
The complexification of $\mathbb{H}$ is defined to be the real tensor product 
between $\mathbb{H}$ itself and $\mathbb{C}$: 
\begin{equation*}
\mathbb{H}_{\mathbb{C}}:=\mathbb{H}\otimes_{\mathbb{R}}\mathbb{C}:=\{p+\imath q\,\mid\, p,q\in\mathbb{H}\}.
\end{equation*}
In $\mathbb{H}_{\mathbb{C}}$ the following associative product is defined: if $p_{1}+\imath q_{1}, p_{2}+\imath q_{2}$ belong to $\mathbb{H}_{\mathbb{C}}$, 
then,
\begin{equation*}
(p_{1}+\imath q_{1})(p_{2}+\imath q_{2})=p_{1}p_{2}-q_{1}q_{2}+\imath(p_{1}q_{2}+q_{1}p_{2}).
\end{equation*}
The usual complex conjugation $\overline{p+\imath q}=p-\imath q$ commutes with
the following involution $(p+\imath q)^{c}=p^{c}+\imath q^{c}$.

We introduce now the class of subsets of $\mathbb{H}$ where our function will be defined.
\begin{definition}
Given any set $D\subseteq\mathbb{C}$, we define its \textit{circularization} as the 
subset in $\mathbb{H}$ defined as follows:
\begin{equation*}
\Omega_{D}:=\{\alpha+I\beta\,\mid\,\alpha+i\beta\in D, I\in\mathbb{S}\}.
\end{equation*}
Such subsets of $\mathbb{H}$ are called \textit{circular} sets. 
If $D\subset \mathbb{C}$ is such that $D\cap\mathbb{R}\neq\emptyset$, then $\Omega_{D}$ is also called a \textit{slice domain} (see \cite{genstostru}). 
\end{definition}

It is clear that, whatever shape the set $D$ has, its circularization $\Omega_{D}$ is
symmetric with respect to the real axis, meaning that, for any 
$x\in\Omega_{D}$ we have that $x^{c}\in\Omega_{D}$.
So, it is not restrictive to start with a set $D$ symmetric with respect to the real line in $\mathbb{C}$.
In particular, if $\alpha+i\beta\in\mathbb{C}$, then $\Omega_{\{\alpha+i\beta\}}=\mathbb{S}_{\alpha+I\beta}$, for any $I\in\mathbb{S}$.


From now on, $\Omega_{D}\subset\mathbb{H}$ will always denote a circular  domain. 

We can state now the
following definition.

\begin{definition}
Let $D\subset\mathbb{C}$ be any symmetric set with respect to the real line. A function $F=F_{1}+\imath F_{2}:D\rightarrow \mathbb{H}_{\mathbb{C}}$ such that $F(\overline z)=\overline{F(z)}$ is said to be a \textit{stem function}. 

\noindent
A function $f:\Omega_{D}\rightarrow\mathbb{H}$ is said to be a \textit{(left) slice
function} if it is induced by a stem function $F=F_{1}+\imath F_{2}$ defined on $D$ 
in the following way: for any $\alpha+I\beta\in\Omega_{D}$,
\begin{equation*}
f(\alpha+I\beta)=F_{1}(\alpha+i\beta)+I F_{2}(\alpha+i\beta).
\end{equation*}

\noindent
If a stem function $F$ induces the slice function $f$, we will write $f=\mathcal{I}(F)$. 
The set of slice functions defined on a certain circular domain $\Omega_{D}$ will
be denoted by $\mathcal{S}(\Omega_{D})$.
Moreover we denote by $\mathcal{S}^{k}(\Omega_{D})$ the set of slice function
of class $\mathcal{C}^{k}$, with $k\in\mathbb{N}\cup\{\infty\}$.
\end{definition}

Notice that $F=F_{1}+\imath F_{2}$ is a stem function if and only if for any $\alpha+i\beta \in D$,
$F_{1}(\alpha-i\beta)=F_{1}(\alpha+i\beta)$ and 
$F_{2}(\alpha-i\beta)=-F_{2}(\alpha+i\beta)$. Then any slice function $f=\mathcal{I}(F_{1}+\imath F_{2})$ is well defined on its domain 
$\Omega_{D}$. If in fact $\alpha+I\beta=\alpha+(-I)(-\beta)\in\Omega_{D}$, then the \textit{even-odd} character of the couple $(F_{1},F_{2})$ grants that 
$f(\alpha+I\beta)=f(\alpha+(-I)(-\beta))$.
 
Given a circular set $\Omega_{D}$ the set $\mathcal{S}^{k}(\Omega_{D})$ is a real vector 
space and also a right $\mathbb{H}$-module for any $k\in\mathbb{N}\cup\{\infty\}$, hence for any $f,g\in\mathcal{S}^{k}(\Omega_{D})$ and for any $q\in\mathbb{H}$, the function 
$f+gq\in\mathcal{S}^k(\Omega_{D})$.

Examples of (left) slice functions are polynomials and power series in the variable $x\in\mathbb{H}$ with all coefficients on the right, i.e.
\begin{equation*}
\sum_{k}x^{k}a_{k},\quad \{a_{k}\}\subset\mathbb{H}.
\end{equation*}

The particular expression of a slice function can be colloquially stated of as a 
quaternionic function of a quaternionic variable that is $\mathbb{H}-$left affine
with respect to the imaginary unit. 
Therefore, the value of a slice
function at any point of its domain $\Omega_{D}$ can be recovered from its values on a 
single slice $\Omega_{D}\cap\mathbb{C}_{I}$ (or two different semislices), for some $I\in\mathbb{S}$. See the \textit{Representation Theorem} in \cite{genstostru, ghiloniperotti}.
%
%
%
%
%

\begin{definition}
Given a slice function $f:\Omega_{D}\rightarrow\mathbb{H}$, the \textit{spherical derivative} of $f$ at $x\in\Omega_{D}\setminus \mathbb{R}$ 
is defined as
\begin{equation*}
\partial_{s}f(x):=\frac{1}{2}Im(x)^{-1}(f(x)-f(x^{c})),
\end{equation*}
while the \textit{spherical value} of $f$ in $x\in\Omega_{D}$ is defined as
 \begin{equation*}
 v_{s}f(x):=\frac{1}{2}(f(x)+f(x^{c})).
 \end{equation*}
\end{definition}

\begin{remark}
Both the spherical derivative and the spherical value of a slice function $f$
are slice functions. In fact, if $f=\mathcal{I}(F_{1}+\imath F_{2})$, 
$x=\alpha+I\beta\in\Omega_{D}$ and $z=\alpha+i\beta\in D$ is the corresponding
point in $\mathbb{C}$, then
$v_{s}f(x)=\mathcal{I}(F_{1}(z))$, while $\partial_{s}f(x)=\mathcal{I}(\frac{F_{2}(z)}{Im(z)})$.
Observe that, given a slice function $f$, its spherical derivative vanishes at $x$ 
if and only if the restriction $f_{|{\mathbb{S}_{x}}}$ is constant. Therefore,
since
the spherical derivative and value are constant on every sphere $\mathbb{S}_{x}$, for any $f\in\mathcal{S}(\Omega_{D})$, it holds
\begin{equation*}
\partial_{s}(\partial_{s}(f))=0\quad\mbox{ and }\quad \partial_{s}(v_{s}(f))=0.
\end{equation*}
%
\end{remark}

\subsubsection{Regularity}

Let now $D\subset \mathbb{C}$ be an open set and $z=\alpha+i\beta\in D$. Given a stem function $F=F_1+\imath F_2:D\rightarrow \hh_\mathbb{C}$
of class $\mathcal{C}^1$, then
\begin{equation*}
 \frac{\partial F}{\partial z},\frac{\partial F}{\partial\bar{z}}:D\rightarrow\hh_\mathbb{C},
\end{equation*}
defined as,
\begin{equation*}
 \frac{\partial F}{\partial z}=\frac{1}{2}\left(\frac{\partial F}{\partial \alpha}-\imath \frac{\partial F}{\partial \beta}\right)\quad \mbox{ and }\quad \frac{\partial F}{\partial \bar{z}}=\frac{1}{2}\left(\frac{\partial F}{\partial \alpha}+\imath \frac{\partial F}{\partial \beta}\right),
 \end{equation*}
are stem functions.
The previous stem functions induce the continuous \textit{slice derivatives}:
\begin{equation*}
\partial_{c} f=\mathcal{I}\left(\frac{\partial F}{\partial z}\right),\quad\overline{\partial}_{c}f=\mathcal{I}\left(\frac{\partial F}{\partial \overline{z}}\right).
\end{equation*}
While the spherical derivative controls the behavior of a slice function $f$ along the 
``spherical'' directions determined by $\mathbb{S}$ (see for instance Corollary 28 of \cite{altavilladiff}), the slice derivatives 
$\partial_{c}$ and $\overline{\partial}_{c}$, give information about the behavior along the
remaining directions (i.e. along the slices).

Now, left multiplication by $\imath $ defines a complex structure on $\mathbb{H}_{\mathbb{C}}$ and, with respect to this structure, a $C^{1}$ stem function $F:D\rightarrow\mathbb{H}_{\mathbb{C}}$
is holomorphic if and only if $\frac{\partial F}{\partial\bar z}\equiv 0$.


We are now in position to define slice regular functions (see Definition 8 in \cite{ghiloniperotti}).
\begin{definition}
Let $\Omega_{D}$ be a circular open set. A function $f=\mathcal{I}(F)\in\mathcal{S}^{1}(\Omega_{D})$ is \textit{(left) regular} if its stem function $F$ is holomorphic. 
The set of regular functions will be denoted by
\begin{equation*}
\mathcal{SR}(\Omega_{D}):=\{f\in\mathcal{S}^{1}(\Omega_{D})\,|\,f=\mathcal{I}(F),F:D\rightarrow \mathbb{H}_{\mathbb{C}}\mbox{ holomorphic}\}.
\end{equation*}
\end{definition}
Equivalently, a slice function $f\in\mathcal{S}^1(\Omega_D)$ is regular if the following equation holds:
 \begin{equation*}
 \overline{\partial}_{c}f(\alpha+J\beta)=0,\quad \forall \,\, \alpha+J\beta\in\Omega_D.
 \end{equation*}
The set of regular functions is again a real vector space and a right $\mathbb{H}$-module.
In the case in which $\Omega_{D}$ is a slice domain, the definition of regularity is equivalent to the one given in \cite{genstostru}.
\begin{remark}
As it is said in Remark 1.6 of \cite{ghilperglobal}, every regular function is real analytic and, moreover, the slice derivative $\partial_{c}f$ of a regular function $f$ is regular on the same domain.
\end{remark}

\begin{remark}
As in the holomorphic case we say that a function $f=\mathcal{I}(F)\in\mathcal{S}^{1}(\Omega_{D})$ is \textit{(left) anti-regular} if its stem function $F$ is anti-holomorphic. Equivalently if
$\partial_{c}f(\alpha+J\beta)=0$, for any $\alpha+J\beta\in\Omega_D$.
\end{remark}

\subsubsection{Product of slice functions and their zero set}

In general, the pointwise product of slice functions is not a slice 
function,
so we need another notion of product.
The following notion is of great importance in the theory and it is, indeed, the one used in the book \cite{genstostru}. The presentation that we are going to use was given in \cite{ghiloniperotti}.
\begin{definition}
Let $f=\mathcal{I}(F)$, $g=\mathcal{I}(G)$ both belonging to $\mathcal{S}(\Omega_{D})$ then the \textit{slice product} of $f$ and $g$ is the slice function
\begin{equation*}
f* g:=\mathcal{I}(FG)\in\mathcal{S}(\Omega_{D}).
\end{equation*}
\end{definition}
Explicitly, if $F=F_{1}+\imath F_{2}$ and $G=G_{1}+\imath G_{2}$ are stem functions, then
\begin{equation*} 
FG=F_{1}G_{1}-F_{2}G_{2}+\imath (F_{1}G_{2}+F_{2}G_{1}).
\end{equation*}
It is now well known that the slice product between two power series in the 
variable $x\in\mathbb{H}$ coincides with their convolution product, i.e.
if $f(x)=\sum_jx^ja_j$ and $g(x)=\sum_kx^kb_k$ are converging power series with coefficients $a_j,b_k\in\mathbb{H}$, then
\begin{equation*}
 (f*g)(x):=\sum_nx^n\left(\sum_{j+k=n}a_jb_k\right).
\end{equation*}
%

\begin{remark}\label{leibnizspherical}
An analogue of the Leibnitz formula holds for the $\partial_{s}$ and $\partial_{c}$ operators: if $f$, $g$ are slice functions then the spherical derivative of their product works as follows:
\begin{equation*}
 \partial_{s}(f* g)=(\partial_{s}f)(v_{s}g)+(v_{s}f)(\partial_{s}g).
\end{equation*}
If $f,g\in\mathcal{SR}(\Omega_{D})$ then $f* g\in \mathcal{SR}(\Omega_{D})$,
 moreover, it holds (see \cite{ghiloniperotti}, Proposition 11):
 \begin{equation*}
 \partial_{c}(f*g)=(\partial_{c}f)*g+f*(\partial_{c}g).
 \end{equation*}

\end{remark}

 The slice product of two slice functions coincides with the punctual product if the first slice function is \textit{slice preserving}.

A slice function $f=\mathcal{I}(F)$ is called \textit{slice-preserving} if, 
for all $J\in \mathbb{S}$, $f(\Omega_{D}\cap\mathbb{C}_{J})\subset \mathbb{C}_{J}$.
Slice preserving functions satisfy the following characterization.
It is well known that, if $f=\mathcal{I}(F_{1}+\imath F_{2})$ is a slice function, then $f$ is slice preserving if and only if the $\mathbb{H}$-valued components $F_{1}$, $F_{2}$ are real valued.
From the definition of slice product and thanks to Proposition \ref{slpreschar}
if $f=\mathcal{I}(F),\, g=\mathcal{I}(G)$ both belong to $\mathcal{S}(\Omega_{D})$, with $f$ slice preserving, then
\begin{equation*}
(f*g)(x)=f(x)g(x).
\end{equation*}

It is now easy to see that if $f$ is a slice preserving function
and $g$ is any slice function, then $fg=f*g=g*f$.
If both $f$ and $g$ are slice preserving, then $fg=f*g=g*f=gf$.
These functions are special since, in a certain sense, transpose the concept of complex
 function in our setting. In fact, if $h(z)=u(z)+iv(z)$ is a complex function defined
 over some domain $D\subset\mathbb{C}$ such that $h(\bar z)=\overline{h(z)}$, then the function $H:D\rightarrow \mathbb{H}_{\mathbb{C}}$ defined as $H(z)=u(z)+\imath v(z)$ is a stem function,
and $\mathcal{I}(H)$ is a slice preserving function.
As stated in \cite{gentilistoppato2}, if
$f$ is a  regular function defined on $\mathbb{B}_{\rho}$, then it is slice preserving if and only if $f$ can be expressed as a power series of the form
\begin{equation*}
f(x)=\sum_{n\in\mathbb{N}}x^{n}a_{n},
\end{equation*}
with $a_{n}$ real numbers.
\newline\\
Given any quaternionic function $f:\Omega\subset\mathbb{H}\rightarrow\mathbb{H}$ of one quaternionic variable we will denote its zero set in the following way:
\begin{equation*}
 \mathcal{Z}(f):=\{x\in\Omega\,|\,f(x)=0\}
\end{equation*}

It is possible to express the slice product of two slice functions in terms of their 
punctual product properly evaluated. The next proposition clarifies this fact; 
its proof can be found in the book \cite{genstostru} and in the context of
stem/slice functions in \cite{altavilla}.

\begin{proposition}\label{prodcomp}
Let $f,g\in\mathcal{SR}(\Omega_{D})$ then, for any $x\in\Omega_{D}\setminus \mathcal{Z}(f)$, 
\begin{equation*}
(f* g)(x)=f(x)g(f(x)^{-1}xf(x)),
\end{equation*}
and $(f* g)(x)=0$ if $f(x)=0$.
\end{proposition}

Given a regular function $f:\Omega_{D}\rightarrow\mathbb{H}$ we will sometimes use the following notation:
\begin{equation*}
T_{f}(x):=f(x)^{-1}xf(x).
\end{equation*}

Recall from \cite{genstostru,ghiloniperotti}, that given any slice function $f=\mathcal{I}(F)\in\mathcal{S}(\Omega_{D})$, then $F^{c}(z)=F(z)^{c}:=F_{1}(z)^{c}+\imath F_{2}(z)^{c}$ is a stem function.
We define the \textit{slice conjugate} of $f$ as the function $f^{c}:=\mathcal{I}(F^{c})\in\mathcal{S}(\Omega_{D})$, while the \textit{symmetrization} of $f$ is defined as $f^{s}:=f^c* f$.
We have that $(FG)^{c}=G^{c}F^{c}$, and so $(f* g)^{c}=g^{c}* f^{c}$, i.e.
$f^{s}=(f^{s})^{c}$.
Moreover it holds
  \begin{equation*}
   (f* g)^{s}=f^{s}g^{s}\quad \mbox{and }\quad (f^{c})^{s}= f^{s}.
  \end{equation*}
Notice that, if $f$ is slice preserving, then $f^{c}=f$ and so $f^{s}=f^{2}$.
We are going now to spend some words on the geometry of the zero locus of a slice function.
First of all, thanks to the Representation Theorem (see \cite{genstostru,ghiloniperotti}), given a slice function $f:\Omega_{D}\rightarrow \mathbb{H}$,
then, for any $x\in\Omega_{D}$, either $\mathcal{Z}(f)\cap\mathbb{S}_{x}=\{y\}$ or $\mathbb{S}_{x}\subset\mathcal{Z}(f)$ or $\mathbb{S}_{x}\cap\mathcal{Z}(f)=\emptyset$.
These three cases
justify the following definition.
\begin{definition}
Let $f:\Omega_{D}\rightarrow \mathbb{H}$ be any slice function with
zero locus $\mathcal{Z}(f)$. Let $x\in\Omega_{D}\cap\mathcal{Z}(f)$ 
be a zero for $f$. We give the following names:
\begin{itemize}
\item if $x\in\mathbb{R}$, then it is called a \textit{real} zero;
\item if $y\notin\mathbb{R}$ and $\mathbb{S}_{y}\cap\mathcal{Z}(f)=\{y\}$,
then $y$ is called an $\mathbb{S}$\textit{-isolated (non-real)} zero;
\item if $x\notin\mathbb{R}$ and $\mathbb{S}_{x}\subset\mathcal{Z}(f)$, then
$x$ is called a \textit{spherical zero}. 
\end{itemize}
\end{definition}

\begin{remark}\label{structpres}
If $f=\mathcal{I}(F)$ is a slice preserving function then it cannot have non-real $\mathbb{S}$-isolated zeros. In fact, since
%
the components of $F$ are real-valued functions, then $0=f(\alpha+I\beta)=F_{1}(\alpha+i\beta)+IF_{2}(\alpha+i\beta)$ if and only if 
$F_{1}(\alpha+i\beta)=0$ and $F_{2}(\alpha+i\beta)=0$ and so 
$f_{|{\mathbb{S}_{\alpha+I\beta}}}\equiv 0$.
\end{remark}

Given now two slice functions $f,g:\Omega_{D}\rightarrow\mathbb{H}$ thanks to Proposition \ref{prodcomp},
it holds,
\begin{equation}\label{zerosofprod}
\mathcal{Z}(f)\subset\mathcal{Z}(f*g),\quad\mbox{ while in general }\quad\mathcal{Z}(g)\not\subset\mathcal{Z}(f*g).
\end{equation}
What is true in general is the following equality:
\begin{equation*}
\bigcup_{x\in \mathcal{Z}(f* g)} \mathbb{S}_{x}=\bigcup_{x\in  \mathcal{Z}(f)\cup  \mathcal{Z}(g)}\mathbb{S}_{x}
\end{equation*}
\begin{example}\label{polexa}
We give now a couple of examples hoping to clarify the previous situations.
Given two generic quaternions $q_{0},q_{1}\in\mathbb{H}$, consider
the quaternionic polynomial $P_{q_{0},q_{1}}:\mathbb{H}\rightarrow\mathbb{H}$ 
defined as $P_{q_{0},q_{1}}(x)=(x-q_{0})*(x-q_{1})=x^{2}-x(q_{0}+q_{1})+q_{0}q_{1}$. This is of course a regular function which vanishes at $q_{0}$ but, in general, not at $q_{1}$. If, in fact, $q_{0}, q_{1}\notin\mathbb{R}$ and $q_{1}\neq q_{0}^{c}$, then the (possibly coincident) roots of $P_{q_{0},q_{1}}$ are $q_{0}$ and
$(q_{1}- q^{c}_{0})^{-1}q_{1}(q_{1}- q^{c}_{0})$ (see section 3.5 of \cite{genstostru}). 
If $q_{l}$, with $l=0,1$ is a real number, then $(x-q_{l})$ is a slice preserving function; therefore, in this case $(x-q_{0})*(x-q_{1})=(x-q_{1})*(x-q_{0})$ and both $q_{0},q_{1}$ are roots. The case in which $q_{1}=q_{0}^{c}$ will be discussed later due to its own importance. Let see now two concrete examples.
\begin{itemize}
\item The polynomial $P_{i,j}(x)=(x-i)*(x-j)=x^{2}-x(i+j)+k$, vanishes only at $i$.
In fact, the second root is given by $(j+i)^{-1}j(j+i)$ that is exactly $i$.
So $i$ is an $\mathbb{S}$-isolated zero for $P_{i,j}$ and it is its only root.
\item The polynomial $P_{	i,2i}(x)=(x-i)*(x-2i)=x^{2}-3xi-2$, vanishes only at $i$ and at $2i$, therefore, $i$ and $2i$ are both $\mathbb{S}$-isolated zeros for $P_{i,2i}$. This polynomial is such that $P_{i,2i}(\mathbb{C}_{i})\subset \mathbb{C}_{i}$, i.e. it preserves the slice $\mathbb{C}_{i}$. Such kind of functions are called \textit{one-slice preserving} and are widely studied in~\cite{altavilladefcj,altavilladefexp}.
\end{itemize}

\end{example}


If we now add regularity, we obtain the following results.

\begin{theorem}[\cite{ghiloniperotti}, Theorem 20]\label{sliceidentity} Let $\Omega_{D}$ be a circular domain. If $f$ is  regular and $f^{s}$ does not vanish identically, then
\begin{equation*}
\mathbb{C}_{J}\cap\bigcup_{x\in \mathcal{Z}(f)}\mathbb{S}_{x}
\end{equation*}
is closed and discrete in $\Omega_{D}\cap\mathbb{C}_{J}$ for all $J\in\mathbb{S}$.
If $\Omega_{D}\cap\mathbb{R}\neq\emptyset$, then $f^{s}\equiv 0$ if and only if $f\equiv 0$.
\end{theorem}

There is also a viceversa, namely, the \textit{Identity principle}~\cite{genstostru, altavilla}], stating that 
given $f=\mathcal{I}(F):\Omega_{D}\rightarrow \mathbb{H}$ a regular function defined over a circular domain
$\Omega_D$, if there exist $K, J\in\mathbb{S}$ with $K\neq J$ such that both the sets $(\Omega_{D}\cap\mathbb{C}_{J}^{+})\cap\mathcal{Z}(f)$ and $(\Omega_{D}\cap\mathbb{C}_{K}^{+})\cap\mathcal{Z}(f)$ admit accumulation points, then $f\equiv 0$ on $\Omega_{D}$.
If $\Omega_{D}$ is a slice domain and if $f:\Omega_{D}\rightarrow \mathbb{H}$ is a slice regular function, then the previous statement
simplifies in the following way:
if there exists $I\in\mathbb{S}$ such that $(\Omega_{D}\cap\mathbb{C}_{I})\cap \mathcal{Z}(f)$ has an accumulation point, then $f\equiv 0$ on $\Omega_{D}$. 

It is also well known~\cite{genstostru}, that if $f\in\mathcal{SR}(\Omega_{D})$ and $\mathbb{S}_{x}\subset\Omega_{D}$ then the zeros of $f^{c}$ on
$\mathbb{S}_{x}$ are in bijective correspondence with those of $f$. Moreover $f^{s}$ vanishes exactly on the sets $\mathbb{S}_{x}$
on which $f$ has a zero.

The next corollary will be used a lot in the next pages. 
We start with a notation: 
given any set $V\subset\mathbb{H}$, we denote by $V_{\mathbb{C}}$, the following set:
\begin{equation*}
V_{\mathbb{C}}:=\{\alpha+i\beta\in\mathbb{C}\,|\,\alpha+I\beta\in V,\mbox{ for some }I\in\mathbb{S}\}\subseteq\mathbb{C}.
\end{equation*}
%
%
%
%
The previous set is constructed so that it takes trace, in the complex plane,
of all the elements of $V$.  Therefore, if, for instance $V=\{3+j,1+j,1+k\}\subset\mathbb{H}$, then $V_{\mathbb{C}}=\{3+i,1+i\}\subset\mathbb{C}$, while, for instance
$V\cap\mathbb{C}_{j}=\{3+j,1+j\}$, $V\cap\mathbb{C}_{k}=\{1+k\}$ and $V\cap\mathbb{C}_{I}=\emptyset$, for any $I\in\mathbb{S}\setminus\{j,k\}$.


\begin{corollary}\label{zerocpt}
Let $\Omega_{D}$ be a slice domain and let $f:\Omega_{D}\rightarrow\mathbb{H}$ be a regular function.
Let $K\subset D$ be a compact set 
containing an accumulation point,
then, either $\mathcal{Z}(f)_{\mathbb{C}}\cap K$ is finite or $f\equiv 0$.
\end{corollary}

\begin{proof}
Let assume that $f\not\equiv 0$ is a regular function such that $\mathcal{Z}(f)_{\mathbb{C}}\cap K$ is not finite. Then, thanks to the first inclusion in equation \eqref{zerosofprod} with $g=f^{c}$, $\mathcal{Z}(f)_{\mathbb{C}}\cap K$ contains a convergent sequence
$\{q_{n}\}_{n\in\mathbb{N}}$ such that $\mathbb{S}_{q_{n}}\subset \mathcal{Z}(f^{s})$, but then, thanks to the Identity Principle, since $\Omega_{D}\cap\mathbb{R}\neq \emptyset$, $f^{s}\equiv 0$ and, thanks to Theorem \ref{sliceidentity}, this is equivalent to $f\equiv 0$.
\end{proof}

Thanks to Remark~\ref{structpres}, we have the following.

\begin{corollary}\label{zerocptcor}
Let $\Omega_{D}$ be a slice domain and let $f:\Omega_{D}\rightarrow\mathbb{H}$ 
be a non-constant slice preserving regular function. Let $\rho>0$ such that $\overline{\mathbb{B}_{\rho}}$ is contained in 
$\Omega_{D}$, then $\mathcal{Z}(f)\cap \mathbb{B}_{\rho}$ and 
$\mathcal{Z}(f)\cap\partial \mathbb{B}_{\rho}$ are finite unions of real points and isolated spheres.
\end{corollary}


The next definition is needed to define the multiplicity of a zero of slice function at a point.
Moreover it provides a set of polynomial functions that will give several information in other
parts of the theory. We already mentioned them in Example \ref{polexa}:  it 
was the case in which $q_{1}=q_{0}^{c}$. References for this set of functions are section 7.2 of \cite{ghiloniperotti}
and the whole paper \cite{ghiloniperotti3}, in which they play a fundamental role.
\begin{definition}
The \textit{characteristic polynomial} of $q$ is the  regular function $(x-q)^{s}:\mathbb{H}\rightarrow\mathbb{H}$ defined by:
 \begin{equation*}
(x-q)^{s}=(x-q)* (x-q^c)=x^2-x(q+q^c)+qq^c.
 \end{equation*}
\end{definition}
\begin{remark}
 The following facts about the characteristic polynomial are quite obvious. If the reader needs more details we refer again to \cite{ghiloniperotti}.
 \begin{itemize}
  \item $(x-q)^{s}$ is a slice preserving function. 
  \item Two characteristic polynomials $(x-q)^{s}$, $(x-q')^{s}$ coincide if and only if $\mathbb{S}_q=\mathbb{S}_{q'}$. 
  \item $\mathcal{Z}((x-q)^{s})=\mathbb{S}_q$. 
  \item From Proposition \ref{prodcomp}, it holds, for any $x\notin\mathbb{S}_{q}$,
  \begin{equation*}
  (x-q)^{s}=(x-q)(T_{(x-q)}(x)-q^{c})=(x-q)((x-q)^{-1}x(x-q)-q^{c})
  \end{equation*}

 \end{itemize}
\end{remark}

Now, from Proposition 3.17 of \cite{genstostru} and Corollary 23 of \cite{ghiloniperotti}, if $f:\Omega_{D}\rightarrow \mathbb{H}$ is a regular function and $q \in \mathcal{Z}(f)$, then there exists $g\in\mathcal{SR}(\Omega_D)$ such that $f(x)=(x-q )* g(x)$. 
Moreover, if $q $ is a spherical zero, then, $(x-q )^{s}$ divides $f$.
Therefore, in both cases, the characteristic polynomial $(x-q )^{s}$ divides $f^{s}$.

We can now recall the following definition (see Definition 14 in \cite{ghiloniperotti}).
\begin{definition}
 Let $f\in\mathcal{SR}(\Omega_D)$ such that $f^{s}$ does not vanish identically. Given $n\in\mathbb{N}$ and $q \in \mathcal{Z}(f)$, we say that $q $ is a zero of $f$
 of \textit{total multiplicity} $n$, and we will denote it by $m_f(q )$,  if $((x-q )^{s})^n\mid f^{s}$ in $\mathcal{SR}(\Omega_{D})$ and $((x-q )^{s})^{n+1}\nmid f^{s}$ in $\mathcal{SR}(\Omega_{D})$. 
 If $m_f(q )=1$, then $q $ is called a \textit{simple zero} of $f$.
\end{definition}

\begin{remark}
If $f\in\mathcal{SR}(\Omega_{D})$ is slice preserving, then, its zeros can only be
real isolated or spherical isolated. Therefore,
if $\{r_{h}\}_{h\in\mathbb{N}}$ is the set of real zeros of $f$, $\{S_{k}\}_{k\in\mathbb{N}}$ the set of 
spheres containing a spherical zero of $f$ (i.e. $f|_{S_{k}}\equiv 0$),
for any $k$, $q_{k}$ is any element in
$S_{k}$, and $\rho>0$ is such that the ball $\overline{\mathbb{B}_{\rho}}$
centered in zero of radius $\rho$, is contained in $\Omega_{D}$, then, 
\begin{equation*}
f_{|{\mathbb{B}_{\rho}}}(x)=x^{n}\left(\prod_{r_{h}\in(\mathcal{Z}(f)\cap \mathbb{B}_{\rho})\cap\mathbb{R}}(x-r_{h})^{n_{h}}\right)\left(\prod_{S_{k}\in(\mathcal{Z}(f)\cap \mathbb{B}_{\rho})\setminus\mathbb{R}}{((x-q_{k})^{s})}^{n_{k}}\right)g(x),
\end{equation*}
where $n,n_{h},n_{k}$ are all positive integers, the products are all finite (thanks to
Corollary \ref{zerocptcor}), and $g$ is a slice preserving regular function which has
no zeros in $\mathbb{B}_{\rho}$. In this situation, since $f^{s}=f^{2}$, then, $m_{f}(0)=2n$,
$m_{f}(q_{l})=2n_{l}$, for any $l$.

Analogues considerations hold for $\partial \mathbb{B}_{\rho}$.
\end{remark}


%
%

\subsubsection{Semiregular functions and their poles}
We will recall now some concept of the theory of meromorphic functions in the context of regularity in the space of quaternions. We will start by introducing the 
concept of slice reciprocal.
Since we are mostly interested in functions defined on euclidean ball
centered in zero of $\mathbb{H}$, the main reference will be the monograph
\cite{genstostru}. However further developments and generalizations on this topic are obtained in \cite{ghilperstosing}. In fact, part of the approach
we are going to use come from this last mentioned paper.

We will first introduce
the notion of reciprocal in the framework of slice functions.
Some material about this notion is collected in \cite{genstostru} and, in more general contexts in \cite{ghilperstosing,altavilla}.

\begin{definition}
Let $f=\mathcal{I}(F)\in\mathcal{SR}(\Omega_{D})$. We call the \textit{slice reciprocal} of $f$ the slice function
\begin{equation*}
f^{-*}:\Omega_{D}\setminus \mathcal{Z}(f^{s})\rightarrow \hh, \qquad f^{-*}=\mathcal{I}((F^cF)^{-1}F^c).
\end{equation*}
\end{definition}
From the previous definition it follows that, if $x\in\Omega_D\setminus \mathcal{Z}(f^{s})$, then
\begin{equation*} 
f^{-*}(x)=(f^{s}(x))^{-1}f^{c}(x).
\end{equation*}
The regularity of the reciprocal just defined follows thanks to the last equality.
The following proposition justify the name \textit{slice reciprocal}.
Observe that, if $f$ is slice preserving, then $f^{c}=f$ and so $f^{-*}=f^{-1}$ where it is defined. Moreover $(f^{c})^{-*}=(f^{-*})^{c}$.

\begin{proposition}[\cite{genstostru, altavilla}]
Let $f\in\mathcal{SR}(\Omega_{D})$ such that $\mathcal{Z}(f)=\emptyset$, then $f^{-*}\in\mathcal{SR}(\Omega_{D})$ and 
\begin{equation*}
f* f^{-*}=f^{-*}* f=1.
\end{equation*}
\end{proposition}

Now, for the general theory of semiregular functions, we refer to \cite{genstostru,ghilperstosing}. Here 
we state the main features needed to our work.
Let $q$ be any quaternion. For any sequence $\{a_{n}\}_{n\in\mathbb{Z}}\subset\mathbb{H}$, the series,
\begin{equation*}
\sum_{n\in\mathbb{Z}}(x-q )^{*n}a_{n},
\end{equation*}
is called the \textit{Laurent series} centered at $q $ associated with $\{a_{n}\}_{n\in\mathbb{Z}}$. In the particular case in which $a_{n}=0$ for any $n<0$, then the previous series is called the \textit{power series} centered at $q$ associated with $\{a_{n}\}_{n\in\mathbb{Z}}$.

Let now $q \in\mathbb{C}_{J}\subset\mathbb{H}$. For any $R,R_{1},R_{2}\in[0,+\infty]$ such that $R_{1}<R_{2}$ we set,
\begin{equation*}
D_{J}(q ,R):=\{z\in\mathbb{C}_{J}\,|\,|z-q |<R\},\quad
A_{J}(q ,R_{1},R_{2}):=\{z\in\mathbb{C}_{J}\,|\, R_{1}<|z-q |<R_{2}\}.
\end{equation*}
Starting from the previous sets we define the following circular one,
\begin{equation*}
\Omega(q ,R):=\bigcup_{J\in\mathbb{S}}D_{J}(q ,R)\cap D_{J}(q ^{c},R),\quad
\Omega(q ,R_{1},R_{2}):=\bigcup_{J\in\mathbb{S}}A_{J}(q ,R_{1},R_{2})\cap A_{J}(q ^{c},R_{1},R_{2}),
\end{equation*}
and set the following notation:
\begin{equation*}
\Sigma(q,R):=\Omega(q,R)\cup D_{J}(q,R),\quad
\Sigma(q,R_{1},R_{2}):=\Omega(q,R_{1},R_{2})\cup A_{J}(q,R_{1},R_{2}).
\end{equation*}

\begin{theorem}[Theorem 4.9, \cite{ghilperstosing}]
Let $q\in\mathbb{H}$, $f\in\mathcal{SR}(\Omega_{D})$ and $0\leq R_{1}\leq R_{2}\leq\infty$ be such that $\Sigma(q ,R_{1},R_{2})\subset\Omega_{D}$. There 
exists a unique sequence $\{a_{n}\}_{n\in\mathbb{Z}}\subset\mathbb{H}$,
such that
\begin{equation}\label{laurentexp}
f(x)=\sum_{n\in\mathbb{Z}}(x-q )^{*n}a_{n},\quad\forall x\in \Sigma(q,R_{1},R_{2}).
\end{equation}
If $\Sigma(q ,R_{2})\subseteq\Omega_{D}$, then for any $n<0$, $a_{n}=0$ and
equation \eqref{laurentexp} holds for any $x\in \Sigma(q ,R_{2})$.
\end{theorem}

We can now state the definition of \textit{pole} and of \textit{semiregularity}.

\begin{definition}
Let $f:\Omega_{D}\rightarrow\mathbb{H}$ be a regular function. A point $q \in\mathbb{H}$
is a \textit{singularity} for $f$ if there exists $R>0$ such that $\Omega_{D}$ contains $\Sigma(q ,0,R)$ and so that the Laurent expansion of $f$ at $q $, 
$f(x)=\sum_{n\in\mathbb{Z}}(x-q )^{* n}a_{n}$, converges in $\Sigma(q ,0,R)$.

Let $q $ be a singularity for $f$. We say that $q $ is a \textit{removable singularity}
if $f$ extends to a neighborhood of $q $ as a regular function. Otherwise consider
the expansion in equation \eqref{laurentexp}: we say that $q $ is a \textit{pole} for $f$ if
there exists $m\leq 0$ such that $a_{-k}=0$ for any $k>m$. The minimum of
such $m$ is called \textit{order of the pole} and denoted by $ord_{f}(q )$. If $q $ is
not a pole, then we call it an \textit{essential singularity} for $f$ and set $ord_{f}(q )=+\infty$.

A function $f:\Omega_{D}\rightarrow\mathbb{H}$ is said to be \textit{semiregular} if it is regular
in some set $\Omega_{D'}\subset\Omega_{D}$ such that every point in 
$\mathcal{P}=\Omega_{D}\setminus\Omega_{D'}$ is a pole or a removable
singularity for $f$.

If a function $f$ is semiregular, then the set of its poles will be denoted by $\mathcal{P}(f)$.
\end{definition}

For more convenience we denote by $\hat{\mathbb{H}}:=\mathbb{H}\cup\{\infty\}$.
In the next pages if a semiregular function $f$ admits a pole at $p$, then we
will write $f(p)=\infty$.

\begin{remark}\label{topsing}
We state here a couple of claims on the topology of $\mathcal{P}(f)$ for any semiregular function $f$ (see \cite{genstostru} and~\cite[Theorem 9.4]{ghilperstosing}).
\begin{itemize}
\item For any imaginary unit $I\in\mathbb{S}$, the set $\mathcal{P}_{I}=\mathcal{P}(f)\cap\mathbb{C}_{I}$ is discrete. 
\item If $f$ is semiregular in $\Omega_{D}$, then $\mathcal{P}(f)$ consists of 
isolated real points and of isolated $2$-spheres of type $\mathbb{S}_{p}$.
\end{itemize}
\end{remark}
%

And now a summary of results about the possibility to represent a semiregular function as the product of some factors.

\begin{proposition}[\cite{genstostru}]\label{semireg}
Let $f,g:\Omega_{D}\rightarrow\mathbb{H}$ be regular functions and consider the
quotient
\begin{equation*}
f^{-*}* g:\Omega_{D}\setminus \mathcal{Z}(f^{s})\rightarrow\mathbb{H}.
\end{equation*}
Each $q \in \mathcal{Z}(f^{s})$ is a pole of order $ord_{f^{-*}* g}(q )\leq m_{f^{s}}(q )$ 
for $f^{-*}* g$. As a consequence the function $f^{-*}* g$ is semiregular on $\Omega_{D}$.
Moreover, for any $x\in\Omega_{D}\setminus \mathcal{Z}(f^{s})$ it holds,
\begin{equation*}
(f^{-*}* g)(x)=f(f^{c}(x)^{-1}xf^{c}(x))^{-1}g(f^{c}(x)^{-1}xf^{c}(x))=f(T_{f^{c}}(x))^{-1}g(T_{f^{c}}(x)).
\end{equation*}
\end{proposition}

Compare the last equation with the one in Proposition \ref{prodcomp}.
Observe that for any regular function $f:\Omega_{D}\rightarrow \mathbb{H}$ and for any $x\in\Omega_{D}\setminus \mathcal{Z}(f^{s})$, it holds $|x|=|f^{c}(x)^{-1}xf^{c}(x)|=|f(x)^{-1}xf(x)|$.


Given $f=\mathcal{I}(F)$ and $g=\mathcal{I}(G)\in\mathcal{S}(\Omega_{D})$, with $g$ slice preserving, then the slice function $h:\Omega_{D}\setminus \mathcal{Z}(g)\rightarrow \mathbb{H}$, 
defined by $h=\mathcal{I}(G^{-1}F)$ is such that $h(x)=\frac{1}{g(x)}f(x)$
and, of course, $h\in\mathcal{SR}(\Omega_{D}\setminus \mathcal{Z}(g))$.

Conversely with respect to the previous proposition, as we will see in the next results, all semiregular functions can be locally expressed as quotients of regular functions. Moreover, if, in the previous statement, $g\equiv 1$, then,
for any semiregular function $f$, its slice inverse $f^{-*}$ is semiregular as well.


\begin{theorem}[\cite{genstostru}]\label{thmdecompreg}
Let $\Omega_{D}$ be a slice domain and $f:\Omega_{D}\rightarrow\hat{\mathbb{H}}$ be a semiregular function. Choose $q =\alpha+I\beta\in\Omega_{D}$, 
set $m=ord_{f}(q )$ and $n=ord_{f}(q ^{c})$ and, without loss of generality suppose 
$m\leq n$. Then, there exist a neighborhood $\Omega_{U}\subset \Omega_{D}$
and a unique regular function $g:\Omega_{U}\rightarrow \mathbb{H}$, such that
\begin{equation*}
f(x)={((x-q )^{s})}^{-n}(x-q )^{*(n-m)}* g(x)
\end{equation*}
in $\Omega_{U}\setminus\mathbb{S}_{q }$. Moreover, if $n>0$ then neither $g(q )$,
nor $g(q ^{c})$ vanishes.

Furthermore, in the same hypotheses, there exists a unique semiregular function $h:\Omega_{D}\rightarrow \hat{\mathbb{H}}$ without poles in $\mathbb{S}_{q }$, such that
\begin{equation*}
f(x)={((x-q )^{s})}^{-n}(x-q )^{*(n-m)}* h(x).
\end{equation*}
In the case in which $n>0$ then neither $h(q )$,
nor $h(q ^{c})$ vanishes.

\end{theorem}

In general, given a slice semiregular function $f$, in each sphere contained in
its domain all the poles have the same order
with the possible exception of one, which may have less order. We 
will see that this is not possible in the case of slice preserving semiregular functions.


\begin{theorem}[\cite{genstostru}]
Let $\Omega_{D}$ be a slice domain and $f:\Omega_{D}\rightarrow\hat{\mathbb{H}}$ be a semiregular function. Suppose $f\not\equiv 0$ and let 
$\mathbb{S}_{q }\subset\Omega_{D}$. There exist $m\in\mathbb{Z}$, 
$n\in\mathbb{N}$, $q_{1},\dots,q_{n}\in\mathbb{S}_{q }$, with $q_{i}\neq\overline q_{i+1}$ for all $i\in\{1,\dots, n\}$ such that
\begin{equation}\label{polesdec}
f(x)={((x-q )^{s})}^{m}(x-q_{1})*(x-q_{2})*\dots*(x-q_{n})* g(x),
\end{equation}
for some semiregular function $g:\Omega_{D}\rightarrow \hat{\mathbb{H}}$ which does
not have neither poles nor zeros in $\mathbb{S}_{q }$.
\end{theorem}

\begin{definition}
Let $f:\Omega_{D}\rightarrow\hat{\mathbb{H}}$ be a semiregular function and 
consider the factorization in equation \eqref{polesdec}. If $m\leq0$, then we say
that $f$ has \textit{spherical order} $-2m$ at $\mathbb{S}_{q}$ and write 
$ord_{f}(\mathbb{S}_{q})=-2m$ (even when $q\in\mathbb{R}$). Whenever $n>0$,
we say that $f$ has \textit{isolated multiplicity} $n$ at $q_{1}$.
\end{definition}


Now, after the summary of the known results we are going to specialize
to the case of slice preserving functions.
\begin{lemma}\label{ordpoleslpr}
Let $f:\Omega_{D}\rightarrow\hat{\mathbb{H}}$ be a slice preserving semiregular function on a symmetric slice domain, then if $q\in\Omega_{D}$ is a pole for  $f$, we have
that $ord_{f}(q')=ord_{f}(q)$, for all $q'\in\mathbb{S}_{q}$.
\end{lemma}
\begin{proof}
Suppose, by contradiction, that there exists a point $q'\in\mathbb{S}_{q}$ such that 
$n=ord_{f}(q')>ord_{f}(q)=m>0$. Without loss of generality we can suppose that
$q'=q^{c}$, then, by Theorem \ref{thmdecompreg}, there exist a neighborhood 
$\Omega_{U}$ of $q$ contained in $\Omega_{D}$ and a unique regular function
$g:\Omega_{U}\rightarrow\mathbb{H}$ such that 
\begin{equation*}
f(x)={((x-q)^{s})}^{-n}(x-q)^{*(n-m)}* g(x)
\end{equation*}
and neither $g(q)$ nor $g(q^{c})$ vanishes.
Now, since $f$ and $(x-q)^{s}$ are slice preserving functions, then,
\begin{equation*}
\tilde f(x):=(x-q)^{*(n-m)}* g(x)={((x-q)^{s})}^{n}f(x)
\end{equation*}
is a slice preserving function. Since $m$ is strictly less than $n$, then $\tilde f(q)=0$ and
$\tilde f(q^{c})\neq 0$ but, since $\tilde f$ is a slice preserving function this is not possible 
and the only possibility is that $m=n$.
\end{proof}

Therefore, if $f\not\equiv 0$ is a slice preserving semiregular function and $\mathbb{S}_{q}\subset\Omega_{D}$ is a spherical pole for $f$, then there exists a negative integer $m$, such that
\begin{equation*}
f(x)={((x-q)^{s})}^{m}g(x),
\end{equation*}
for some slice preserving semiregular function $g:\Omega_{D}\rightarrow\mathbb{H}$
which does not have poles nor zeros in $\mathbb{S}_{q}$.

%

We now want to state an analogue of Corollary \ref{zerocpt} in the case of poles.

\begin{proposition}\label{polecpt}
Let $\Omega_{D}$ be a slice domain and let $f:\Omega_{D}\rightarrow\hat{\mathbb{H}}$ be a non-constant semiregular function.
Let $K\subset D$ be a compact set 
containing an accumulation point,
then, $\mathcal{P}(f)_{\mathbb{C}}\cap K$ is finite.
\end{proposition}

\begin{proof}
The proof is just a consequence of Remark \ref{topsing}, Theorem \ref{thmdecompreg} and Corollary \ref{zerocpt}. If in fact $f$ is semiregular on $\Omega_{D}$, then
$f^{-*}$ is semiregular in $\Omega_{D}$. Moreover, $\mathbb{S}_{x}\cap\mathcal{P}(f)\neq\emptyset$ if and only if $\mathbb{S}_{x}\cap\mathcal{Z}(f^{c})\neq\emptyset$ and $\mathbb{S}_{x}\cap\mathcal{Z}(f)\neq\emptyset$ if and only if $\mathbb{S}_{x}\cap\mathcal{P}(f^{-*})\neq\emptyset$.
\end{proof}

Again we specialize now to the case of balls.
\begin{corollary}\label{polecptcor}
Let $\Omega_{D}$ be a slice domain and let $f:\Omega_{D}\rightarrow\hat{\mathbb{H}}$ 
be a non-constant slice preserving semiregular function. Let $\rho>0$ such that the 
closed ball centered in zero with radius $\rho$, $\overline{\mathbb{B}_{\rho}}$, is contained in 
$\Omega_{D}$, then $\mathcal{P}(f)\cap \overline{\mathbb{B}_{\rho}}$ and 
$\mathcal{P}(f)\cap\partial \overline{\mathbb{B}_{\rho}}$ are finite unions of real points and isolated spheres.
\end{corollary}

We end this subsection collecting all we need for the last parts of this paper.
\begin{remark}
Let $\Omega_{D}$ be a slice domain and $f:\Omega_{D}\rightarrow\hat{\mathbb{H}}$
be a slice preserving semiregular function. Let $\mathcal{ZP}(f):=\mathcal{Z}(f)\cup\mathcal{P}(f)$ be the set of zeros and poles of $f$, then, 
if $\{r_{h}\}_{h\in\mathbb{N}}$ is the set of real zeros and poles of $f$, $\{S_{k}\}_{k\in\mathbb{N}}$ the set of spherical zeros and poles, for any $k$, $q_{k}$ is any element in
$S_{k}$, and $\rho>0$ is such that the closed ball $\overline{\mathbb{B}_{\rho}}$
centered in zero of radius $\rho$, is contained in $\Omega_{D}$, then, 
\begin{equation*}
\boxed{
f_{|{\mathbb{B}_{\rho}}}(x)=x^{n}\left(\prod_{r_{h}\in(\mathcal{ZP}(f)\cap \mathbb{B}_{\rho})\cap\mathbb{R}}(x-r_{h})^{n_{h}}\right)\left(\prod_{S_{k}\in(\mathcal{ZP}(f)\cap \mathbb{B}_{\rho})\setminus\mathbb{R}}{((x-q_{k})^{s})}^{n_{k}}\right)g(x),
}
\end{equation*}
where $n,n_{h},n_{k}$ are all integers, the products are all finite (thanks to
Corollaries \ref{zerocptcor} and \ref{polecptcor}), and $g$ is a slice preserving regular function which has
no zeros nor poles in $\mathbb{B}_{\rho}$. 
\end{remark}

\subsection{Quaternionic $\rho$-Blaschke factors}

In this subsection we are going to reproduce some results proved in \cite{irenebla,shur} for a modification of quaternionic Blaschke factors.

\begin{definition}
Given $\rho>0$ and $a\in\mathbb{H}$ such
that $|a|<\rho$. We define the $\rho$\textit{-Blaschke factor} at $a$ 
as the following semiregular function:
\begin{equation*}
B_{a,\rho}:\mathbb{H}\rightarrow\hat{\mathbb{H}},\quad
B_{a,\rho}(x):=(\rho^{2}-x a^{c})*(\rho(x-a))^{-*}.
\end{equation*}
If now $\rho>0$ and $a\in\mathbb{H}\setminus\mathbb{R}$ is such that $|a|<\rho$, we define the $\rho$-Blaschke factor at the sphere $\mathbb{S}_{a}$ as the following slice preserving semiregular function:
\begin{equation*}
B_{\mathbb{S}_{a},\rho}:\mathbb{H}\rightarrow\hat{\mathbb{H}},\quad
B_{\mathbb{S}_{a},\rho}(x):=B_{a,\rho}^{s}(x).
\end{equation*}
\end{definition}

The previous definition makes sense thanks to Proposition \ref{semireg}. 
Moreover, in a more explicit form, we have:

%
\begin{equation*}
\begin{array}{rcl}
B_{\mathbb{S}_{a},\rho}(x):=B_{a,\rho}^{s}(x)& = & ((\rho^{2}-x a^{c})*(\rho(x-a))^{-*})^{s}\\
& = & (\rho^{2}-x a^{c})*(\rho(x-a))^{-*}*((\rho^{2}-x a^{c})*(\rho(x-a))^{-*})^{c}\\
& = & (\rho^{2}-x a^{c})*(\rho(x-a))^{-*}*((\rho(x-a))^{-*})^{c}*(\rho^{2}-x a^{c})^{c}\\
& = & (\rho^{2}-x a^{c})*((\rho(x-a))^{c}*\rho(x-a))^{-*}*(\rho^{2}-x a^{c})^{c}.
\end{array}
\end{equation*}
Observe that the central factor, is such that,
\begin{equation*}
(\rho(x-a))^{c}*\rho(x-a))=\rho^{2}(x-a)^{s}
\end{equation*}
and so it is a slice preserving function. Therefore, 
\begin{equation*}
B_{\mathbb{S}_{a},\rho}(x):=B_{a,\rho}^{s}(x)=(\rho^{2}(x-a)^{s})^{-1}(\rho^{2}-x a^{c})*(\rho^{2}-xa).
\end{equation*}
Moreover, since,
\begin{equation*}
\begin{array}{rlll}
(\rho^{2}-x a^{c})*(\rho^{2}-xa) & = & (x a^{c}-\rho^{2})*(xa-\rho^{2})& =((x-\rho^{2} (a^{c})^{-1}) a^{c})*(xa-\rho^{2})\\
& = & (x-\rho^{2} (a^{c})^{-1})* a^{c}*(xa-\rho^{2})
& =  (x-\rho^{2} (a^{c})^{-1})*(x|a|^{2}- a^{c} \rho^{2})\\
& = & (x-\rho^{2} (a^{c})^{-1})*(xa-\rho^{2}) a^{c}
& =  (x-\rho^{2} (a^{c})^{-1})*(x-\rho^{2}a^{-1})|a|^{2}\\
& = & (x-\rho^{2}a^{-1})^{s}|a|^{2}, &
\end{array}
\end{equation*}
then
\begin{equation*}
B_{\mathbb{S}_{a},\rho}(x)=(\rho^{2}(x-a)^{s})^{-1}(x-\rho^{2}a^{-1})^{s}|a|^{2}.
\end{equation*}

\begin{remark}\label{zerobla}
The $\rho$-Blaschke factor at $a$ has only a zero at $\rho^{2} (a^{c})^{-1}$ and a pole at the sphere $\mathbb{S}_{a}$ (collapsing to a point when $a\in\mathbb{R}$), while the $\rho$-Blaschke factor at $\mathbb{S}_{a}$ has a spherical zero at $\mathbb{S}_{\rho^{2} a^{-1}}$ and a pole at the sphere $\mathbb{S}_{a}$.
\end{remark}

We now expose a result similar to Theorem 5.5 of \cite{irenebla}.

\begin{theorem}\label{borderbla}
Given $\rho>0$  and $a\in\mathbb{H}$. The
$\rho$-Blaschke factors $B_{a,\rho}$ and $B_{\mathbb{S}_{a},\rho}$ have the following properties:
\begin{itemize}
\item they satisfy $B_{a,\rho}(\mathbb{H}\setminus\mathbb{B}_{\rho})\subset \mathbb{B}$, $B_{\mathbb{S}_{a},\rho}(\mathbb{H}\setminus\mathbb{B}_{\rho})\subset \mathbb{B}$ and $B_{a,\rho}(\mathbb{B}_{\rho})\subset \mathbb{H}\setminus\mathbb{B}$, $B_{\mathbb{S}_{a},\rho}(\mathbb{B}_{\rho})\subset \mathbb{H}\setminus\mathbb{B}$;
\item they send the boundary of the ball  
$\partial\mathbb{B}_{\rho}$ in the boundary of the ball $\partial \mathbb{B}$.
\end{itemize}
\end{theorem}
\begin{proof}
We prove the result for the $\rho$-Blaschke factor $B_{a,\rho}$, the proof
for the other $B_{\mathbb{S}_{a},\rho}$ goes analogously.
Since $a$ and $ a^{c}$ lie in the same slice, then,
$B_{a,\rho}(x):=(\rho^{2}-x a^{c})*(\rho(x-a))^{-*}=(\rho(x-a))^{-*}*(\rho^{2}-x a^{c})$. 
Hence, thanks to Proposition \ref{semireg}, for any
$x\in\mathbb{H}\setminus\mathbb{S}_{a}$,
there exists $\tilde x\in\mathbb{S}_{x}$, such that 
\begin{equation*}
|B_{a,\rho}(x)|^{2}=|(\rho^{2}-\tilde x a^{c})|^{2}|\rho(\tilde x-a)|^{-2}.
\end{equation*}
Therefore $|B_{a,\rho}|^{2}<1$ if and only if $|(\rho^{2}-\tilde x a^{c})|^{2}<|\rho(\tilde x-a)|^{2}$ and this is equivalent to,
\begin{equation*}
\rho^{4}+|x|^{2}|a|^{2}<\rho^{2}(|x|^{2}+|a^{2}|).
\end{equation*}
But now, the last inequality is equivalent to say $(\rho^{2}-|x|^{2})(\rho^{2}-|a|^{2})<0$ and this is possible if and only if $\rho^{2}<|x|^{2}$.

For the second part of the theorem, repeat the previous computations observing  that imposing $|B_{a,\rho}(x)|^{2}=1$ is equivalent to $|x|=\rho$.
\end{proof}

\subsection{PQL functions}

In this subsection we want to introduce a family of quaternionic functions of one
quaternionic variable which will be part of the subject of what follows.
\begin{definition}
Let $\{q_{k}\}_{k=1}^{N}\subset\mathbb{H}$, $\{a_{k}\}_{k=0}^{N}\subset\mathbb{H}\setminus\{0\}$ be finite sets of, possibly repeated quaternions
and fix, for any $k=1,\dots, N$, $M_{k}\in\{\pm 1\}$. 
A function $f:\mathbb{H}\rightarrow\hat{\mathbb{H}}$ is said to be a \textit{PQL function} if it is given by:
\begin{equation*}
f(x)=a_{0}\prod_{k=1}^{N}(x-q_{k})^{M_{k}}a_{k}.
\end{equation*}
\end{definition}

The class of PQL functions is not contained in the class of slice regular functions. If, in fact we consider two non real quaternions $q_{0}, q_{1}$,
and consider the following two PQL functions: $f_{1}(x):=x-q_{0}$ and $f_{2}(x):=(x-q_{0})(x-q_{1})$. Then, obviously,
$f_{1}$ is regular but $f_{2}(x)=x^{2}-xq_{1}-q_{0}x+q_{0}q_{1}$ is not.
\begin{example}
A particular subclass of PQL-functions is the class of {\it linear fractional transformations} of the extended quaternionic space $\mathbb{H} \cup \{ \infty \} \cong \mathbb{HP}^1$. 
Recalling that $GL( 2, \mathbb{H})$ denotes the group of $ 2 \times 2$ invertible quaternionic matrices, one way to represent linear fractional transformations is the following:
$$
\mathbb{G} = \left\{ g(x)=(ax+b)(cx+d)^{-1} \,\, | \,\, \begin{bmatrix}  
                                                                                          a & b \\
                                                                                          c & d \\
                                                                                      \end{bmatrix} 
                                                                                      \in GL( 2, \mathbb{H} )   \right\}.
$$
It is well known that $\mathbb{G}$ forms a group with respect to the composition operation. Denotes now $SL(2, \mathbb{H})$ as the subgroup of the matrices of $GL( 2, \mathbb{H})$ with Dieudonn\'e determinant equal to 1. 
Moreover, the linear fractional transformation $g(x)=(ax+b)(cx+d)^{-1}$ is constant iff the Dieudonn\'e determinant of the associated matrix $\begin{bmatrix}
																																																																																								a & b \\
																																																																																								c & d 
																																																																																								\end{bmatrix} $
is zero. \\
The group $\mathbb{G}$ is isomorphic to $PSL( 2 , \mathbb{H}) =SL(2, \mathbb{H}) / \{ \pm Id \}$ and to $GL(2, \mathbb{H}) / \{ k\cdot  Id \},$ where $k \in \mathbb{R} \setminus \{ 0 \} ;$ all the elements in $\mathbb{G}$ are conformal maps: for a proof of these facts, for a definition of Dieudonn\'e determinant and for more
details see \cite{bisigentili}, \cite{bisigentilitrends} and \cite{bisistoppato}. 
The group $\mathbb{G}$ is generated with usual composition by the following four types of transformations:
\begin{center}
\begin{tabular}{ l l }
 i) $L_1 (x)=x+b, \,\, b \in \mathbb{H}$; & iii) $L_3 (x)= r  x=x  r, \,\, r \in \mathbb{R}^{+} \setminus \{ 0 \}$;  \\ 
 & \\
 ii) $L_2 (x)= x  a, \,\, a \in \mathbb{H}, \,\, |a|=1$; & iv) $L_4 (x)=x^{-1}$. \\   
\end{tabular}
\end{center}

Well studied is also the subgroup $\mathbb{M}$ of $\mathbb{G}$ of the so called {\it M\"obius transformations} mapping the quaternionic open unit ball $\mathbb{B}$ onto itself. This is defined as follows:

$$
Sp(1,1)= \{ C \in GL(2, \mathbb{H})\,\, | \,\, \overline{C}^t H C = H \} \subset SL(2, \mathbb{H}), \,\mbox{ where }\, H= \begin{bmatrix}
           1 & 0 \\
           0 & -1 \\
           \end{bmatrix},
$$
and an element $g \in \mathbb{G}$ is a M\"obius transformation in $\mathbb{M}$ if and only if $g(x)=(ax +b)  (cx+d)^{-1}$ with $\begin{bmatrix}
                                                                                                                                                                                                         a & b \\
                                                                                                                                                                                                         c & d \\
                                                                                                                                                                                                \end{bmatrix} \in Sp(1,1).$
This is equivalent to
$$g(x)= v (x-q_0)(1-\overline{q}_0 x)^{-1} u^{-1}$$ 
for some $u,v \in \partial \mathbb{B},$ and $q_0 \in \mathbb{B}$. Observe the
link between this function and the following punctual $1$-Blaschke functions: in a certain sense functions like the previous $g$ are reciprocal of the one defined below if $u$ and $v$ are both equal to $1$.
\end{example}

\begin{remark}\label{NivenP}
Notice that any PQL functions with positive exponents $M_{k}$ is a Niven polynomial~\cite{niven}, i.e. a function of the type $a_{0}xa_{1}x\dots xa_{n}+\phi(x)$, where $a_{0},a_{1},\dots, a_{n}$ are non-zero quaternions and $\phi$ is a sum of a finite number of similar functions $b_{0}xb_{1}x\dots xb_{k}$, $k<n$. Such functions satisfy a ``fundamental theorem of algebra''. Moreover, the opposite is not true, meaning that not all Niven polynomials can be written as
a PQL function and a counterexample is given by the function $h(q)=-iq^2i +(i+1)q(i+1)^{-1}$.
\end{remark}

%
%
%


\begin{definition}
Given $\rho>0$ and $a\in\mathbb{H}$ such
that $|a|<\rho$. We define the punctual $\rho$\textit{-Blaschke} factor at $a$ to be the PQL function
\begin{equation*}
B_{a,\rho}^{p}:\mathbb{H}\rightarrow\hat{\mathbb{H}},\quad
B_{a,\rho}^{p}(x):=(\rho^{2}-x a^{c})(\rho(x-a))^{-1}.
\end{equation*}

\end{definition}

Analogously (even in the proof) of Theorem \ref{borderbla} we have the following result.

\begin{theorem}\label{borderblapun}
Given $\rho>0$ and $a\in\mathbb{B}_{\rho}\subset\mathbb{H}$, $|a|<\rho.$ The
punctual $\rho$-Blaschke factors $B_{a,\rho}^{p}$ have the following properties:
\begin{itemize}
\item they satisfy $B_{a,\rho}^{p}(\mathbb{H}\setminus\mathbb{B}_{\rho})\subset \mathbb{B}$ and $B_{a,\rho}^{p}(\mathbb{B}_{\rho})\subset \mathbb{H}\setminus\mathbb{B}$;

\item they send the boundary of the ball 
$\partial\mathbb{B}_{\rho}$ onto the boundary of the ball $\partial \mathbb{B}$.
\end{itemize}
\end{theorem}

\subsection{Quaternionic holomorphic functions}

This last subsection contains the actual starting point of the original part of this work.

First of all, if we represent a quaternion $x$ as $x_{0}+ix_{1}+jx_{2}+kx_{3}$,
we can define the following two quaternionic differential operators, 
\begin{equation*}
\overline{\mathcal{D}}_{CF}=\frac{\partial}{\partial x_{0}}-i\frac{\partial}{\partial x_1}-j\frac{\partial}{\partial x_2}-k\frac{\partial}{\partial x_3},\quad
\mathcal{D}_{CF}=\frac{\partial}{\partial x_{0}}+i\frac{\partial}{\partial x_1}+j\frac{\partial}{\partial x_2}+k\frac{\partial}{\partial x_3}.
\end{equation*}
The previous two operators are called \textit{Cauchy-Fueter} operators. 

 \begin{definition}
 Let $\Omega\subset\mathbb{H}$ be a domain.
  A quaternionic function of one quaternionic variable $f:\Omega\rightarrow\mathbb{H}$ of class $\mathcal{C}^{3}$ is 
  said to be (left) \textit{quaternionic holomorphic} if it satisfies the following equation
 \begin{equation}\label{quahol}
\mathcal{D}_{CF}\Delta f(x)=0,
 \end{equation}
where $x=x_{0}+x_1i+x_2j+x_3k$, and $\Delta$ denotes the usual laplacian in the four real variables $x_{0}, x_1, x_2, x_3$.
%
 \end{definition}

The key observation, now, is the fact that quaternionic polynomials and converging power series of the variable $x$ with coefficients on the right are contained in the set of quaternionic holomorphic functions (see \cite{fueter}).
Now, even if we have seen quickly how to expand semiregular functions
in power and Laurent series, these not always have euclidean open sets of convergence.
Nevertheless a different type of series expansion has been studied so far, namely the \textit{spherical power series} \cite{genstostru, ghiloniperotti3, ghilperstosing},
and it admits actual euclidean open domains as domains of convergence. For this reason the following proposition, already observed in 
the PhD thesis of the first author \cite{altavillaPhD}, holds true.

\begin{proposition}[\cite{altavillaPhD}]
Any  regular function $f:\Omega_{D}\rightarrow\mathbb{H}$ is quaternionic holomorphic.
Moreover, since $f$ satisfies equation \eqref{quahol}, then it also satisfies the following equation:
\begin{equation*}
 \Delta\Delta f=0.
\end{equation*}
\end{proposition}
The last equality holds because $\overline{\mathcal{D}}_{CF}\mathcal{D}_{CF}=\mathcal{D}_{CF}\overline{\mathcal{D}}_{CF}=\Delta$.
For more details about  the theory of quaternionic holomorphic functions, we refer to 
\cite{fueter, laville,pernas}.


\begin{remark}
Not all PQL functions are quaternionic holomorphic: it is sufficient
to consider the function $h(x)=x(x-i)(x-j)x=x^{4}-x[k-ix-xj]x$. The function $x^{4}$ is quaternionic holomorphic,
on the other hand, $\mathcal{D}_{CF}\Delta (x[k-ix-xj]x)\neq 0$. 
\end{remark}

As pointed out by Perotti in [\cite{perotti}, Section 6], 
for any class $\mathcal{C}^{1}$ slice function $f$, the following two
formulas hold true:
\begin{equation}\label{cfslice}
\boxed{
\mathcal{D}_{CF}f-2\overline{\partial}_{c}f=-2\partial_{s}f,\quad \overline{\mathcal{D}}_{CF}f-2\partial_{c} f=2\partial_{s}f}
\end{equation}
Moreover he stated the following theorem. 
\begin{theorem}[\cite{perotti}]\label{agacse}
Let $\Omega_{D}$ be any circular domain and
let $f:\Omega_{D}\subset\mathbb{H}\rightarrow\mathbb{H}$ be a slice function of
class $\mathcal{C}^{1}(\Omega_{D})$, then
$f$ is regular if and only if $\mathcal{D}_{CF}f=-2\partial_{s}f$.
Moreover it holds $2\partial_{c}\partial_{s}f=\overline{\mathcal{D}}_{CF}\partial_{s}f$.

If now $f$ is regular, then, 
$\mathcal{D}_{CF}\Delta f=\Delta\mathcal{D}_{CF} f=-2\Delta \partial_{s} f=0$,
i.e. the spherical derivative of a regular function is harmonic. Moreover
$\partial_{c}\partial_{s}f$  is harmonic too.
\end{theorem}

In the sequel, we will widely use these recalled results.\\
On the other hand, if the reader wants to see another approach to harmonicity for slice regular functions, we recommend the enlighten reading of \cite{BW}.

\section{Log-biharmonicity and Riesz measure}

We start now with a notation:
given a semiregular function $f:\Omega_{D}\rightarrow\hat{\mathbb{H}}$, we remember
the following set:
\begin{equation*}
\mathcal{ZP}(f)=\mathcal{Z}(f)\cup \mathcal{P}(f)\subset\Omega_{D}.
\end{equation*}

The first fundamental result is the following.
\begin{theorem}\label{slicepreslogbi}
Let $f=\mathcal{I}(F):\Omega_{D}\rightarrow\mathbb{H}$ be a slice preserving semiregular function. Then,
\begin{equation*}
\Delta^{2}\log|f|^{2}=0,\quad \forall x\in\Omega_{D}\setminus\mathcal{ZP}(f).
\end{equation*}
\end{theorem}
\begin{proof}
Since, as we have already mentioned $\Delta=\mathcal{D}_{CF}\overline{\mathcal{D}}_{CF}$, if we prove that $\overline{\mathcal{D}}_{CF}\log|f|^{2}$ is
a regular function outside $\mathcal{ZP}(f)$, then we have the thesis.
Now,
\begin{equation*}
\overline{\mathcal{D}}_{CF}\log|f|^{2}=\frac{1}{|f|^{2}}\overline{\mathcal{D}}_{CF}|f|^{2},
\end{equation*}
but since $f$ is slice preserving, then 
$|f|^{2}=f f^{c}=f* f^{c}=\mathcal{I}((F_{1}+\imath F_{2})(F_{1}-\imath F_{2}))=F_{1}^{2}+F_{2}^{2}$
is a slice preserving function such that
$\partial_{s}|f|^{2}=0$. Therefore, by Theorem \ref{agacse}, $\overline{\mathcal{D}}_{CF}|f|^{2}=2\partial_{c}|f|^{2}=2\partial_{c}f  f^{c}$.
At this point the thesis follows from the next computation:
\begin{equation*}
\bar\partial_{c}\left(\frac{(\partial_{c}f) f^{c}}{|f|^{2}}\right)=-\frac{1}{|f|^{4}}f(\overline{\partial_{c}}f^{c})(\partial_{c}f) f^{c}+\frac{1}{|f|^{2}}(\partial_{c}f)(\overline{\partial_{c}}f^{c})=0.
\end{equation*}
\end{proof}

\begin{definition}
Given $f:\Omega\rightarrow \mathbb{H}$ of class $\mathcal{C}^{\infty}$ we say that
$f$ has \textit{log-biharmonic modulus} if
$$\Delta^{2}\log|f|\equiv 0.$$
\end{definition}

\begin{corollary}\label{laplog}
Let $f:\Omega_{D}\rightarrow\mathbb{H}$ be a slice preserving semiregular function.
Then, on $x\in\Omega_{D}\setminus \mathcal{ZP}(f)$ the following formula holds:
\begin{equation}\label{logbislicepres}
\Delta\log|f|=\frac{1}{2}\Delta\log|f|^{2}=\frac{1}{2}\mathcal{D}_{CF}\left(2\frac{(\partial_{c}f) f^{c}}{|f|^{2}}\right)=-2\frac{\partial_{s}((\partial_{c}f) f^{c})}{|f|^{2}}.
\end{equation}
In particular, if $f(x)=x$, we get that 
\begin{equation*}
\Delta\log|x|=2/|x|^{2}\quad \forall\, x\neq 0.
\end{equation*}
\end{corollary}
\begin{proof}
The thesis follows thanks to the particular form that the Cauchy-Fueter  operators take in the
setting of slice functions (see equation \eqref{cfslice}), and from the computations 
in the proof of Theorem \ref{slicepreslogbi}.
\end{proof}

\begin{remark}\label{generalremarks}
A direct consequence of the previous theorem is the fact that, for any $x\in\mathbb{H}\setminus \{0\}$,
\begin{equation*}
\Delta^{2}\log|x|=0.
\end{equation*}
Starting from this, for any orthogonal transformation or translation $T$ of $\mathbb{R}^{4}$,
since $\Delta(f\circ T)=(\Delta f)\circ T$ (see \cite{axler}, Chapter 1), we also have that
$\Delta^{2}(f\circ T)=(\Delta^{2} f)\circ T$.
Therefore, for any fixed $q_{0}\in\mathbb{H}$, the function $\log|x-q_{0}|$ is biharmonic
for any $x\in\mathbb{H}\setminus\{q_{0}\}$.

Moreover, since $\log(ab)=\log a+\log b$ over the reals, then for any  slice preserving regular function $f$ and for
any quaternion $q_{0}$, the function $|f*(x-q_{0})|$ is log-biharmonic outside of its zeros.

Furthermore, if we set $C:\mathbb{H}\rightarrow\mathbb{H}$, to be the map,
such that $C(x)=x^{c}$, then again we have, by simple computations, that $\Delta (f\circ C)=(\Delta(f))\circ C$
and so $\Delta^{2}(f\circ C)=(\Delta^{2} f)\circ C$. 

\end{remark}

From the previous remark, it makes sense to state the following well known result (see, for instance, \cite{polyharmonic, mitrea}).

\begin{theorem}[Fundamental solution for the bilaplacian in $\mathbb{H}$]\label{delta}

The following equality holds:
\begin{equation*}
\Delta^{2}\left(-\frac{1}{48}\log|x|\right)=\delta_{0},\,\,\,\mbox{for }x\in\mathbb{H},
\end{equation*}
where $\delta_{0}$ denotes the Dirac measure centered in zero.
\end{theorem}
\begin{proof}
The proof is well known and quite standard, however, for sake of completeness
and to justify the coefficient $-48$ appearing in our formula, we will show the 
main steps. First of all, notice that $\log|x|^{2}$ is a radial function, therefore it is useful
to pass to 4D-spherical coordinates $(r,\vartheta)=(r,\vartheta_{1},\vartheta_{2},\vartheta_{3})$, where $r=|x|$.
In these coordinates the laplacian is of the form
\begin{equation*}
\Delta=\frac{\partial^{2}}{\partial r^{2}}+\frac{3}{r}\frac{\partial}{\partial r}+\mathcal{L}(\vartheta),
\end{equation*}
where $\mathcal{L}(\vartheta)$ is the \textit{angular} part of the laplacian. 
The main idea of the proof is to apply the following corollary of the dominated
convergence theorem:\\

\framebox{\parbox{\dimexpr\linewidth-2\fboxsep-2\fboxrule}{\itshape%
\noindent\textbf{Claim}:\\
If $\varphi:\mathbb{R}^{4}\rightarrow \mathbb{R}$ is any positive function such that $\int_{\mathbb{R}^{4}}\varphi=1$, then the family of functions depending on $\epsilon$, 
\begin{equation*}
\varphi_{\epsilon}(x)=\frac{1}{\epsilon^{4}}\varphi\left(\frac{x}{\epsilon}\right),
\end{equation*}
is such that $\int_{\mathbb{R}^{4}}\varphi_{\epsilon}=1$ for any $\epsilon$, and converges
in the sense of distributions to the Dirac delta $\delta_{0}$ for $\epsilon\to 0$.}}

\vspace{0.3cm}
\noindent We will now fix any $\epsilon\in\mathbb{R}$ and compute 
\begin{equation*}
\Delta^{2}\log(|x|^{2}+\epsilon^{2})=\left(\frac{\partial^{2}}{\partial r^{2}}+\frac{3}{r}\frac{\partial}{\partial r}\right)^{2}\log(|x|^{2}+\epsilon^{2}).
\end{equation*}
After standard computations, we get:
\begin{equation*}
\Delta \log(|x|^{2}+\epsilon^{2})=4\frac{r^{2}+2\epsilon^{2}}{(r^{2}+\epsilon^{2})^{2}}
\end{equation*}
\begin{equation*}
\Delta^{2}\log(|x|^{2}+\epsilon^{2})=\Delta\left(4\frac{r^{2}+2\epsilon^{2}}{(r^{2}+\epsilon^{2})^{2}}
\right)=-96\frac{\epsilon^{4}}{(r^{2}+\epsilon^{2})^{4}}=\frac{-96}{\epsilon^{4}}\frac{1}{(\frac{r^{2}}{\epsilon^{2}}+1)^{4}}.
\end{equation*}
If we define $\varphi(|x|):=(|x|^{2}+1)^{-4}$, this is an integrable function for which, we can apply the previous Claim and so,
\begin{equation*}
\Delta^{2}\log|x|^{2}=\lim_{\epsilon \to 0}\Delta^{2}\log(|x|^{2}+\epsilon^{2})=-96\delta_{0}.
\end{equation*}
Therefore, we obtain the thesis:
\begin{equation*}
\Delta^{2}\log|x|=-48\delta_{0}
\end{equation*}

\end{proof}

Since it will recur often, from now on, the coefficient $-48$ of the previous
theorem will be denoted in the following way:
\begin{equation*}
\gamma=-48
\end{equation*}

\begin{corollary}\label{translation}
For any $q\in\mathbb{H}$, the following formulas hold:
\begin{equation*}
\boxed{
\Delta^{2}\log|x-q|=\gamma\delta_{q}
}
\qquad 
\boxed{
\Delta^{2}\log|x^{c}-q|=\gamma\delta_{q^{c}}
}
\end{equation*}

Moreover, for any set $\{q_{k}\}_{k=1}^{N}\subset\mathbb{H}$ and any $\{a_{k}\}_{k=0}^{N}\subset\mathbb{H}\setminus\{0\}$ if we define
$f:\mathbb{H}\rightarrow\mathbb{H}\cup\{\infty\}$ to be the PQL function
defined as $f(x)=a_{0}(x-q_{1})^{M_{1}}a_{1}(x-q_{2})^{M_{2}}\cdots(x-q_{N})^{M_{N}}a_{N}$, with $M_{k}=\pm 1$, then
\begin{equation*}
\Delta^{2}\log|f(x)|=\gamma\sum_{k=1}^{N}M_{k}\delta_{q_{k}}.
\end{equation*}
\end{corollary}
\begin{proof}
The first formula follows from Theorem \ref{delta} and Remark \ref{generalremarks}.   The second one from the first one plus logarithm general properties.
\end{proof}
This corollary and Remark \ref{generalremarks} tell us that the whole theory
can be generalized to \textit{anti}-regular functions and to the analogous of PQL functions in which the variable $x^{c}$ appears accompanied to $x$.


\begin{example}
In the previous corollary, if $a,b,c,d$ are quaternions, such that $a$ or $c$ are not simultaneously zero,
$\begin{bmatrix}  
                                                                                          a & b \\
                                                                                          c & d \\
                                                                                      \end{bmatrix} 
                                                                                      \in Sp(1,1)$
and $g$ is the M\"{o}bius function $g(x)=(ax+b)(cx+d)^{-1}$, then
\begin{equation*}
\frac{1}{\gamma}\Delta^{2}\log|g|=\delta_{-a^{-1}b}-\delta_{-c^{-1}d}.
\end{equation*}
\end{example}


\begin{remark}\label{counterexample}
In general, given a semiregular function $f$, its modulus is not log-biharmonic. In fact, given two distinct non-real quaternions $q_{0}, q_{1}$ and $q_{1}\neq  q_{0}^{c}$, the regular polynomial $P_{q_{0},q_{1}}$ defined in Example \ref{polexa},
has not a log-biharmonic modulus, i.e. for $x\notin\{ q_{0}$, $(q_{1}- q_{0}^{c})^{-1}q_{1}(q_{1}- q_{0}^{c})\}$, in general, we have that
\begin{equation*}
\Delta^{2}\log|P_{q_{0},q_{1}}(x)|\neq 0
\end{equation*}
In particular, we have computed the previous quantity in two particular but significative cases. If $q_{0}=i$ and $q_{1}=j$, then $P_{i,j}$
has one isolated zero on $\mathbb{S}$ and is nowhere else zero. 
It holds
\begin{equation*}
\Delta^{2}\log|P_{i,j}(x)|_{|{x=0}}=64.
\end{equation*}
If, instead, $q_{0}=i$ and $q_{1}=2i$, then $P_{i,2i}$ has two isolated zeros on $\mathbb{C}_{i}$ and is such that $P_{i,2i}(\mathbb{C}_{i})\subseteq\mathbb{C}_{i}$. 
In this particular case $P_{i,2i}(x)=(x-i)(T_{(x-i)}(x)-2i)$ and the function $T_{(x-i)}$ restricted only to $\mathbb{C}_{i}$ is equal to the identity, that is, 
for any $z\in\mathbb{C}_{i}$,
\begin{equation*}
P_{i,2i}(z)=(z-i)(z-2i)
\end{equation*}

Nevertheless,
even this function has not log-biharmonic modulus outside of its zeros, in fact,
even if $0\in\mathbb{C}_{i}$,
\begin{equation*}
\Delta^{2}\log|P_{i,2i}(x)|_{|{x=0}}=-72.
\end{equation*}
These two quantities were computed with the help of the software Mathematica 10 using, instead of the quaternionic variable, its four real coordinates, 
i.e. if $x=x_{0}+x_{1}i+x_{2}j+x_{3}k$, then
\begin{align*}
P_{i,j}(x)&=x_{0}^{2}-x_{1}^{2}-x_{2}^{2}-x_{3}^{2}+x_{1}+x_{2}+(2x_{0}x_{1}-x_{0}+x_{3})i+\\
&+(2x_{0}x_{2}-x_{0}-x_{3})j+(2x_{0}x_{3}-x_{1}+x_{2}+1)k,\\
P_{i,2i}(x)&=x_{0}^{2}-x_{1}^{2}-x_{2}^{2}-x_{3}^{2}-2+3x_{1}+(2x_{0}x_{1}-3x_{0})i+\\
&+(2x_{0}x_{2}-3x_{3})j+(2x_{0}x_{3}+3x_{2})k.
\end{align*}
Hence
\begin{align*}
|P_{i,j}(x)|^{2}&=(x_{0}^{2}-x_{1}^{2}-x_{2}^{2}-x_{3}^{2}+x_{1}+x_{2})^{2}+(2x_{0}x_{1}-x_{0}+x_{3})^{2}+\\
&+(2x_{0}x_{2}-x_{0}-x_{3})^{2}+(2x_{0}x_{3}-x_{1}+x_{2}+1)^{2},\\
|P_{i,2i}(x)|^{2}&=(x_{0}^{2}-x_{1}^{2}-x_{2}^{2}-x_{3}^{2}-2+3x_{1})^{2}+(2x_{0}x_{1}-3x_{0})^{2}+\\
&+(2x_{0}x_{2}-3x_{3})^{2}+(2x_{0}x_{3}+3x_{2})^{2}.
\end{align*}

\end{remark}

After all these remarks, it seems really interesting to study the log-biharmonicity
of the modulus of slice preserving regular functions. The following result goes in this direction 
and gives new genuine information on a class of regular functions.

\begin{theorem}\label{lebesgue}
For any $p\in\mathbb{H}\setminus\mathbb{R}$,
\begin{equation*}
\boxed{
\frac{1}{\gamma}\Delta^{2}\log|(x-p)^{s}|=\mathcal{L}_{\mathbb{S}_{p}},
}
\end{equation*}
where $\mathcal{L}_{\mathbb{S}_{p}}$ denotes the Lebesgue  measure of 
the sphere $\mathbb{S}_{p}$.
\end{theorem}
\begin{proof}
Let $p=\alpha_{0}+I_{0}\beta_{0}$ be any non-real quaternion, then, for any $x\in\mathbb{H}\setminus\mathbb{S}_{p}$, by Theorem \ref{slicepreslogbi} we have,
\begin{equation*}
\Delta^{2}\log|(x-p)^{s}|=0.
\end{equation*}
Therefore the previous equation remains true if we restricts it to any semi-slice $\mathbb{C}_{I}^{+}$, that is, if $\mathbb{S}_{p}\cap\mathbb{C}_{I}^{+}=\{\tilde p\}$ then, for any $x\in\mathbb{C}_{I}^{+}\setminus\{\tilde p\}$,
\begin{equation}\label{eqcharpol}
(\Delta^{2}\log|(x-p)^{s}|) _{|{\mathbb{C}_{I}^{+}}}=0.
\end{equation}
Now, since $\tilde p\in\mathbb{S}_{p}$, then, $(x-p)^s=(x-\tilde{p})^s$,
therefore, 
\begin{equation*}
\Delta^{2}\log|(x-\tilde p)^s| =\Delta^{2}\log|x-\tilde p| +\Delta^{2}\log|T_{(x-\tilde p)}(x)-\tilde p^{c}| =\gamma\delta_{\tilde p} +\Delta^{2}\log|T_{(x-\tilde p)}(x)-\tilde p^{c}|.
\end{equation*}
The last equality, restricted to $\mathbb{C}_{I}^{+}$, gives

\begin{equation*}
(\Delta^{2}\log|(x-\tilde p)^s|) _{|{\mathbb{C}_{I}^{+}}}=\gamma\delta_{\tilde p} +(\Delta^{2}\log|T_{(x-\tilde p)}(x)-\tilde p^{c}|)_{|{\mathbb{C}_{I}^{+}}},
\end{equation*}
where the equality is in the sense of measures.
Taking into account equation~\eqref{eqcharpol}, we obtain the following
\begin{equation*}
(\Delta^{2}\log|(x-\tilde p)^s|) _{|{\mathbb{C}_{I}^{+}}}=\gamma\delta_{\tilde p}.
\end{equation*}

Recall that $(\mathbb{H}\setminus \mathbb{R})\simeq\mathbb{C}_{I}^{+}\times \mathbb{S}$. Denote
by $dz_{I}^{+}$ and $d\sigma_{\mathbb{S}}$ the standard surface measure of $\mathbb{C}_{I}^{+}$ and $\mathbb{S}$, respectively.
If we now take any real valued compactly supported $\mathcal{C}^{\infty}$ function $\varphi$,
we have that,
\begin{equation*}
\int_{\mathbb{H}}\Delta^{2}\log|(x-p)^s|\varphi(x)dx=\int_{\mathbb{S}}\left(\int_{\mathbb{C}_{I}^{+}}\Delta^{2}\log|(x-p)^s|\varphi(x)dz_{I}^{+}\right)d\sigma_{\mathbb{S}}=\gamma\int_{\mathbb{S}}\varphi(\alpha_{0}+I\beta_{0})d\sigma_{\mathbb{S}},
\end{equation*}
where the first equality holds thanks to Fubini's Theorem and the fact that in a neighborhood of the real line
the integrand is measurable and $\mathbb{R}$ has zero measure with respect
to the $4$-dimensional Lebesgue  measure. 
\end{proof}

We end this section with the following summarizing theorem, that allow us to
define the Riesz measure of a slice preserving semiregular function and of a 
PQL function.


\begin{theorem}[Riesz Measure]\label{riesz}
Let $\Omega$ be a domain of $\mathbb{H}$ and let $f:\Omega\rightarrow\hat{\mathbb{H}}$ be a quaternionic function of one quaternionic variable and let $\rho>0$ such that the ball $\overline{\mathbb{B}_{\rho}}\subset\Omega$ and such that $f(y)\neq 0, \infty$, for any $y\in\partial \mathbb{B}_{\rho}$.
\begin{enumerate}[label=(\roman*)]
\item If $f$ is a slice preserving semiregular function such that
 $\{r_{k}\}_{k=1,2,..}$, $\{p_{h}\}_{h=1,2,..}$ are the sets of its real zeros and poles, respectively, and  
$\{\mathbb{S}_{a_{i}}\}_{i=1,2,..}$, $\{\mathbb{S}_{b_{j}}\}_{j=1,2,..}$
are the sets of its spherical zeros and poles, respectively, everything repeated accordingly to their multiplicity, 
then, for any $x\in \mathbb{B}_{\rho}$,
\begin{equation*}
\frac{1}{\gamma}\Delta^{2}\log|f|=\sum_{|r_{k}|<\rho}\delta_{r_{k}}-\sum_{|p_{h}|<\rho}\delta_{p_{h}}+\sum_{|a_{i}|<\rho}\mathcal{L}_{\mathbb{S}_{a_{i}}}-\sum_{|b_{j}|<\rho}\mathcal{L}_{\mathbb{S}_{b_{j}}}.
\end{equation*}
\item  If $f$ is a PQL function
\begin{equation*}
f(x):=a_{0}\prod_{k=1}^{N}(x-q_{k})^{M_{k}}a_{k},
\end{equation*}
with $M_{k}=\pm 1$ and $|q_{k}|<\rho$, for any $k$, then,
\begin{equation*}
\frac{1}{\gamma}\Delta^{2}\log|f|=\sum_{k=1}^{N}M_{k}\delta_{q_{k}}.
\end{equation*}
\end{enumerate}
\end{theorem}
\begin{proof}
The proof of this theorem is just a collection of Theorems \ref{delta}, \ref{lebesgue}
and Remark \ref{translation}.
\end{proof}

The equalities in the previous theorem, must be interpreted in the sense of distributions. Therefore, if $\varphi$ is a $\mathcal{C}^{\infty}_{c}(\mathbb{B}_{\rho})$-function,
then, in case (i),
\begin{equation*}
\begin{array}{rcl}
\displaystyle\int_{\mathbb{B}_{\rho}}\displaystyle\frac{1}{\gamma}\Delta^{2}\log|f(x)|\varphi(x)dx & = & \displaystyle\sum_{|r_{k}|<\rho}\varphi(r_{k})-\displaystyle\sum_{|p_{h}|<\rho}\varphi(p_{h})+\\
& &\\
 & & +\displaystyle\sum_{|a_{i}|<\rho}\displaystyle\int_{\mathbb{S}}\varphi(\alpha_{i}+I\beta_{i})d\sigma_{\mathbb{S}}-\displaystyle\sum_{|b_{j}|<\rho}\displaystyle\int_{\mathbb{S}}\varphi(\alpha_{j}+I\beta_{j})d\sigma_{\mathbb{S}},
\end{array}
\end{equation*}
where, of course, $a_{i}=\alpha_{i}+K_{i}\beta_{i}$ and $b_{j}=\alpha_{j}+K_{j}\beta_{j}$ with $K_{i},K_{j}\in\mathbb{S}$, while in case (ii),
\begin{equation*}
\displaystyle\int_{\mathbb{B}_{\rho}}\displaystyle\frac{1}{\gamma}\Delta^{2}\log|f(x)|\varphi(x)dx =  \displaystyle\sum_{k=1}^{N}M_{k}\varphi(q_{k}).
\end{equation*}

\begin{remark}
Due to Remark \ref{counterexample}, the optimal family of semiregular functions
for which is possible to define the Riesz measure in the sense of the present paper is exactly the slice preserving one.
\end{remark}
\begin{remark}\label{mix1}
Thanks to the properties of the real logarithm, a mix of cases $(i)$ and $(ii)$ in the previous theorem, can be considered. If for instance, $f$ is a slice preserving 
semiregular function as in case (i) of Theorem \ref{riesz} and $q_{1},q_{2}\in \mathbb{B}_{\rho}$,
then, the function $h(x)=(x-q_{1})f(x)(x-q_{2})^{-2}$ is such that
\begin{equation*}
\frac{1}{\gamma}\Delta^{2}\log|h|=\frac{1}{\gamma}\Delta^{2}\log|f|+\delta_{q_{1}}-2\delta_{q_{2}}.
\end{equation*}
\end{remark}
\begin{remark}
Observe that two different functions can give rise to the same Riesz measure.
In fact, for instance, given $q_{1},q_{2}\in\mathbb{H}$, then, 
\begin{equation*}
\frac{1}{\gamma}\Delta^{2}\log|(x-q_{1})(x-q_{2})|=\frac{1}{\gamma}\Delta^{2}\log|(x-q_{2})(x-q_{1})|,
\end{equation*}
even if the two functions $(x-q_{1})(x-q_{2})$ and $(x-q_{2})(x-q_{1})$ are in general pointwise different.
\end{remark}

\begin{example}
To construct the Riesz measure of  the $\rho$-Blaschke $B_{\mathbb{S}_{a},\rho}$, we have just
to remember where its zero and pole are located. But this was observed in Remark \ref{zerobla}, therefore,
\begin{equation*}
\frac{1}{\gamma}\Delta^{2}\log|B_{\mathbb{S}_{a},\rho}|=\mathcal{L}_{\mathbb{S}_{\rho^{2}a^{-1}}}-\mathcal{L}_{\mathbb{S}_{a}}.
\end{equation*}

\end{example}


\section{Jensen formulas and corollaries}
In this section we will present an analogue of the Jensen formula for some classes of quaternionic functions, namely the same considered in Theorem \ref{riesz}.

From now on, $y=y_0+iy_1+jy_2+ky_3$ will be a new quaternionic variable that we will use when necessary.
Let $\mathbb{B}(x,\rho)$ denotes the euclidean ball centered in $x$ with radius $\rho$. Observe that if $x=0$, then $\mathbb{B}(x,\rho)=\mathbb{B}_{\rho}$.
If $\Omega$ is an open set, a necessary and sufficient condition for a function $u:\Omega\subset\mathbb{R}^{4}\rightarrow\mathbb{R}$ to be bihamonic, is to satisfy the following mean value property (see Theorem 7.24 of \cite{mitrea}):\\

\framebox{\parbox{\dimexpr\linewidth-2\fboxsep-2\fboxrule}{\itshape%
\noindent For any $x\in\Omega$ and for any $\rho>0$ such that $\overline{\mathbb{B}(x,\rho)}\subset\Omega$,
\begin{equation*}
u(x)=\frac{1}{|\partial \mathbb{B}(x,\rho)|}\int_{\partial \mathbb{B}(x,\rho)}u(y)d\sigma(y)-\frac{\rho^{2}}{8}\Delta u(x),
\end{equation*}}}\\

where $|\partial \mathbb{B}(x,\rho)|$ denotes the Lebesgue measure of
the boundary of the four dimensional ball of radius $\rho$.

We have seen that, a slice preserving semiregular function or a PQL function $f:\Omega\rightarrow\hat{\mathbb{H}}$, has log-biharmonic modulus outside
of its zeros and singularities, therefore, for any $x\in\Omega\setminus\mathcal{ZP}(f)$ and
for any $\rho>0$ such that $\overline{\mathbb{B}(x,\rho)}\subset\Omega$,
\begin{equation*}
\boxed{
\log|f(x)|=\frac{1}{|\partial \mathbb{B}(x,\rho)|}\int_{\partial \mathbb{B}(x,\rho)}\log |f(y)|d\sigma(y)-\frac{\rho^{2}}{8}\Delta\log |f(x)|.}
\end{equation*}

Starting from this property we are going to prove an analogue of the Jensen
formula. To do that, we need the following technical
lemma which gives new results on $\rho$-Blaschke factors.

%

\begin{lemma}\label{laplmoebius}
Let $\rho>0$, then
\begin{enumerate}
\item[(i)] if $a$ is any non-zero quaternion such that $|a|<\rho$ and $B_{a,\rho}^{p}$
denotes the punctual $\rho$-Blaschke function at $a$, then
\begin{equation*}
\Delta\log|B^{p}_{a,\rho} (x)|_{|{x=0}}=\frac{2}{\rho^{4}|a|^{2}}[|a|^{4}-\rho^{4}],
\end{equation*}
\item[(ii)] if $a$ is any non-real quaternion such that $|a|<\rho$ and $B_{\mathbb{S}_{a},\rho}$
denotes the regular $\rho$-Blaschke function at $a$, then
\begin{equation*}
\Delta\log|B_{\mathbb{S}_{a},\rho} (x)|_{|{x=0}}=\frac{2}{\rho^{4}|a|^{4}}[\rho^{4}-|a|^{4}](2|a|^{2}-(a+ a^{c})^{2})
\end{equation*}
\end{enumerate}
%
%
%
\end{lemma}

\begin{proof}
(i)
We start with the computation for $B_{a,\rho}^p$. First observe that,
\begin{equation*}
\Delta\log|B_{a,\rho}^p(x)|=\Delta\log|\rho^{2}- xa^{c}|-\Delta\log|\rho(x-a)|=\Delta\log|\rho^{2} (a^{c})^{-1}-x|-\Delta\log|x-a|.
\end{equation*}
Thanks to Remark \ref{generalremarks} and Corollary \ref{laplog}, we get,
\begin{equation*}
\Delta\log|B_{a,\rho}^p(x)|=2\left(\frac{1}{|\rho^{2} (a^{c})^{-1}-x|^{2}}-\frac{1}{|x-a|^{2}}\right).
\end{equation*}
Evaluating the last equality in zero, we obtain the thesis,
\begin{equation*}
\Delta\log|B_{a,\rho}^p (x)|_{|{x=0}}= 2\left(\frac{1}{|\rho^{2} (a^{c})^{-1}|^{2}}-\frac{1}{|a|^{2}}\right)=\frac{2}{\rho^{4}|a|^{2}}[|a|^{4}-\rho^{4}].
\end{equation*}
(ii) For the other (regular) $\rho$-Blaschke function $B_{\mathbb{S}_{a},\rho}$, we begin,
as before, by splitting the logarithm of the norm of the ratio
in the difference of the logarithms:
\begin{equation*}
\begin{array}{rcl}
\Delta\log|B_{\mathbb{S}_{a},\rho}(x)| & = & \Delta\log|(\rho^{2}-xa)*(\rho^{2}-x a^{c})|-
\Delta\log|\rho^{2}(x-a)^s|\\
& = & \Delta\log|\rho^{4}-x\rho^{2}(a+ a^{c})+x^{2}|a|^{2}|-\Delta\log|(x-a)^s|.
\end{array}
\end{equation*}
We apply formula \eqref{logbislicepres} to our functions.
Starting from the first part we have that:
\begin{equation*}
|\rho^{4}-x\rho^{2}(a+ a^{c})+x^{2}|a|^{2}|^{2}_{|{x=0}}=\rho^{8}.
\end{equation*}
Then it holds
\begin{equation*}
\partial_{c}(\rho^{4}-x\rho^{2}(a+ a^{c})+x^{2}|a|^{2})=2x|a|^{2}-\rho^{2}(a+ a^{c}),
\end{equation*}
and so, 
\begin{equation*}
\Delta\log|\rho^{4}-x\rho^{2}(a+ a^{c})+x^{2}|a|^{2}|_{|{x=0}}=\frac{-2}{\rho^{8}}\partial_{s}((2x|a|^{2}-\rho^{2}(a+ a^{c}))(\rho^{4}- x^{c}\rho^{2}(a+ a^{c})+ (x^{c})^{2}|a|^{2}))_{|{x=0}}.
\end{equation*}
Following in this direction, we obtain that,
\begin{equation*}
\begin{array}{c}
\partial_{s}((2x|a|^{2}-\rho^{2}(a+ a^{c}))(\rho^{4}- x^{c}\rho^{2}(a+ a^{c})+ (x^{c})^{2}|a|^{2}))_{|{x=0}}=\\
\\
 =\partial_{s}(2x|a|^{2}\rho^{4}-2|x|^{2}|a|^{2}\rho^{2}(a+ a^{c})+2 x^{c} |x|^{2}|a|^{4}-\rho^{6}(a+ a^{c})+ x^{c}\rho^{4}(a+ a^{c})^{2}- (x^{c})^{2}\rho^{2}(a+ a^{c})|a|^{2})_{|{x=0}}=\\
 \\
=2|a|^{2}\rho^{4}-\rho^{4}(a+ a^{c})^{2},
\end{array}
\end{equation*}
%

where the last equality holds thanks to the linearity of the spherical derivative,
the fact that $\partial_{s}x=1$, $\partial_{s} x^{c}=-1$ and the fact that
any real-valued function has zero spherical derivative.
In fact, if $x=\alpha+I\beta$, from the former facts and from Remark \ref{leibnizspherical}, one has that 
\begin{equation*}
\partial_{s}|x|^{2}=\partial_{s}(x x^{c})=\alpha-\alpha=0,\quad \partial_{s}( x^{c} |x|^{2})=-|x|^{2},\quad \partial_{s} (x^{c})^{2}=-2\alpha,
\end{equation*}
and for $x=0$ all vanish.
Collecting everything, we obtain that,
\begin{equation*}
\Delta\log|\rho^{4}-x\rho^{2}(a+ a^{c})+x^{2}|a|^{2}|_{|{x=0}}=\frac{-2}{\rho^{4}}(2|a|^{2}-(a+ a^{c})^{2}).
\end{equation*}
Working on the other term, we obtain:
\begin{equation*}
|(x-a)^{s}|^{2}_{|{x=0}}=|a|^{4},\quad \partial_{c}(x-a)^{s}=2x-(a+ a^{c}).
\end{equation*}
As before, we compute:
\begin{equation*}
\begin{array}{c}
\partial_{s}((\partial_{c}(x-a)^{s})\overline{(x-a)^{s}})=\partial_{s}((2x-(a+ a^{c}))( (x^{c})^{2}- x^{c}(a+ a^{c})+|a|^{2}))_{|{x=0}}=\\ 
\\=\partial_{s}(2 x^{c}|x|^{2}-2|x|^{2}(a+ a^{c})+2x|a|^{2}- (x^{c})^{2}(a+ a^{c})+ x^{c}(a+ a^{c})^{2}-(a+ a^{c})|a|^{2})_{|{x=0}}=\\
\\
=2|a|^{2}-(a+ a^{c})^{2},
\end{array}
\end{equation*}
and so,
\begin{equation*}
\Delta\log|(x-a)^{s}|_{|{x=0}}=\frac{-2}{|a|^{4}}(2|a|^{2}-(a+ a^{c})^{2}).
\end{equation*}
Collecting everything we have the thesis:
\begin{equation*}
\Delta\log|B_{\mathbb{S}_{a},\rho} (x)|_{|{x=0}}=\frac{2}{\rho^{4}|a|^{4}}[\rho^{4}-|a|^{4}](2|a|^{2}-(a+ a^{c})^{2})
\end{equation*}

\end{proof}

\begin{remark}
The two equalities in the last statement are consistent, meaning that, if in the
second equality we consider the limit for $a=a_{0}+Ia_{1}$ that goes to $a_{0}$, with $a_{0}\neq 0$, then
\begin{multline*}
\lim_{a\to a_{0}}\Delta\log|B_{\mathbb{S}_{a},\rho} (x)|_{|{x=0}}=\lim_{a\to a_{0}}\frac{2}{\rho^{4}|a|^{4}}[\rho^{4}-|a|^{4}](2|a|^{2}-(a+ a^{c})^{2})=\\
=\frac{2}{\rho^{4}a_{0}^{4}}[\rho^{4}-a^{4}](-2a_{0}^{2})=\frac{4}{\rho^{4}a_{0}^{2}}[a_{0}^{4}-\rho^{4}]=2\Delta\log|B_{a_{0},\rho}^p (x)|_{|{x=0}}=\Delta\log|B_{a_{0},\rho}^p (x)|_{|{x=0}}^{2}.
\end{multline*}
\end{remark}

As for Theorem \ref{riesz}, we will state now an analogue of the 
Jensen formula for the classes of functions we are dealing with.
\begin{theorem}[Jensen formulas]\label{jensen}
Let $\Omega$ be a domain of $\mathbb{H}$ and let $f:\Omega\rightarrow\hat{\mathbb{H}}$ be a quaternionic function of one quaternionic variable and let $\rho>0$ such that the ball $\overline{\mathbb{B}_{\rho}}\subset\Omega$, $f(0)\neq 0,\infty$ and such that $f(y)\neq 0, \infty$, for any $y\in\partial \mathbb{B}_{\rho}$.
\begin{enumerate}[label=(\roman*)]
\item If $f$ is a slice preserving semiregular function such that
 $\{r_{k}\}_{k=1,2,..}$, $\{p_{h}\}_{h=1,2,..}$ are the sets of its real zeros and poles, respectively, and  
$\{\mathbb{S}_{a_{i}}\}_{i=1,2,..}$, $\{\mathbb{S}_{b_{j}}\}_{j=1,2,..}$
are the sets of its spherical zeros and poles, respectively, everything repeated accordingly to their multiplicity. 
Then,
\begin{equation}\label{jensenreg}
\begin{array}{rcl}
\log|f(0)| & = & \displaystyle\frac{1}{|\partial \mathbb{B}_{\rho}|}\displaystyle\int_{\partial \mathbb{B}_{\rho}}\log|f(y)|d\sigma(y)-\displaystyle\frac{\rho^{2}}{8}\Delta\log|f (x)|_{|{x=0}}+\\
& & \\
& & +\displaystyle\sum_{|p_{h}|<\rho}\left(\log\displaystyle\frac{\rho}{|p_{h}|}+\displaystyle\frac{1}{4}\left(\displaystyle\frac{p_{h}^{4}-\rho^{4}}{\rho^{2}p_{h}^{2}}\right)\right) -\displaystyle\sum_{|r_{k}|<\rho}\left(\log\displaystyle\frac{\rho}{|r_{k}|}+\displaystyle\frac{1}{4}\left(\displaystyle\frac{r_{k}^{4}-\rho^{4}}{\rho^{2}r_{k}^{2}}\right)\right)\\
& & \\
& & +\displaystyle\sum_{|b_{j}|<\rho}\left(\log\displaystyle\frac{\rho^{2}}{|b_{j}|^{2}}+\displaystyle\frac{1}{4}\left(\displaystyle\frac{\rho^{4}-|b_{j}|^{4}}{\rho^{2}|b_{j}|^{4}}\right)[2|b_{j}|^{2}-(b_{j}+ b^{c}_{j})^{2}]\right)\\
& & \\
& & -\displaystyle\sum_{|a_{i}|<\rho}\left(\log\displaystyle\frac{\rho^{2}}{|a_{i}|^{2}}+\displaystyle\frac{1}{4}\left(\displaystyle\frac{\rho^{4}-|a_{i}|^{4}}{\rho^{2}|a_{i}|^{4}}\right)[2|a_{i}|^{2}-(a_{i}+ a^{c}_{i})^{2}]\right).
\end{array}
\end{equation}

\item  If $f$ is a PQL function
\begin{equation*}
f(x):=a_{0}\prod_{k=1}^{N}(x-q_{k})^{M_{k}}a_{k},
\end{equation*}
with $M_{k}=\pm 1$ and $|q_{k}|<\rho$, for any $k$, then,

\begin{equation}\label{jensenpql}
\begin{array}{rcl}
\log|f(0)| & = & \displaystyle\frac{1}{|\partial \mathbb{B}_{\rho}|}\displaystyle\int_{\partial \mathbb{B}_{\rho}}\log|f(y)|d\sigma(y)-\displaystyle\frac{\rho^{2}}{8}\Delta\log|f|_{|{x=0}}+\\
& & \\
& & -\displaystyle\sum_{k=1}^{N}M_{k}\left(\log\displaystyle\frac{\rho}{|q_{k}|}+\displaystyle\frac{1}{4}\left(\displaystyle\frac{|q_{k}|^{4}-\rho^{4}}{\rho^{2}|q_{k}|^{2}}\right)\right).
\end{array}
\end{equation}

\end{enumerate}
\end{theorem}
The basic idea of the proof comes from the complex case (see for instance \cite{ahlfors, jank, rudin}). Starting from that, we expose the details in our quaternionic setting.

%

\begin{proof}
First of all, since $\overline{\mathbb{B}_{\rho}}\subset\Omega$, then $\{r_{k},p_{h}\}_{k,h=1,2,\dots}\cap\mathbb{B}_{\rho}$ is finite and $\{\mathbb{S}_{a_{i}},\mathbb{S}_{b_{j}}\}_{i,j=1,2,\dots}\cap\mathbb{B}_{\rho}$  is a finite set of spheres
(see Corollaries \ref{zerocptcor} and \ref{polecptcor}).

To begin, suppose that $f$ has no zeros or poles in $\mathbb{B}_{\rho}$. Then, since $\log|f|$ is biharmonic
in $\mathbb{B}_{\rho}$, then formulas \eqref{jensenreg} and \eqref{jensenpql} are exactly the mean value property for biharmonic functions.

Suppose now that $f$ is a slice preserving semiregular function such that   $\mathcal{Z}(f)=\{r_{k},\mathbb{S}_{a_i}\}_{k,i=1,2,\dots}$ and $\mathcal{P}(f)=\{p_{h},\mathbb{S}_{b_j}\}_{h,j=1,2,\dots}$ repeated according to their multiplicity and $f(0)\neq 0,\infty$. Define $g$ as the following function:

%

\begin{equation}\label{poissonstart}
g(x)  :=  \left(\prod_{|p_{h}|<\rho} B^{p}_{p_{h},\rho} (x) \prod_{|b_{j}|<\rho}B_{\mathbb{S}_{b_{j}},\rho} (x) \right)^{-1}
 \left(\prod_{|r_{k}|<\rho} B^{p}_{r_{k},\rho} (x) \prod_{|a_{i}|<\rho}B_{\mathbb{S}_{a_{i}},\rho}(x) \right) f(x).
\end{equation}

Observe that each factor on the right hand side is a slice preserving semiregular function.
Moreover, $g(x)$ is different from $0$ and $\infty$ in $|x|<\rho$, hence $\log |g(x)|$
is a biharmonic function and so it satisfies the biharmonic mean value property:
\begin{equation*}
\log |g(0)|=\frac{1}{|\partial \mathbb{B}_{\rho}|}\int_{\partial \mathbb{B}_{\rho}}\log|g(y)|d\sigma(y)-\frac{\rho^{2}}{8}\Delta \log|g (x)|_{|{x=0}},
\end{equation*}
but 
\begin{equation*}
|g(0)|=|f(0)|\left(\prod_{|p_{h}|<\rho}\frac{\rho }{|p_{h}|}\right)^{-1}\left(\prod_{|b_{j}|<\rho}\frac{\rho^{2}}{|b_{j}|^{2}}\right)^{-1}\left(\prod_{|r_{k}|<\rho}\frac{\rho}{|r_{k}|}\right)\left(\prod_{|a_{i}|<\rho}\frac{\rho^{2}}{|a_{i}|^{2}}\right)
\end{equation*}
and then
\begin{equation*}
\log|g(0)|=\log|f(0)|+2\left[\displaystyle\sum_{|b_{j}|<\rho}\log\displaystyle\frac{\rho}{|b_{j}|}-\displaystyle\sum_{|a_{i}|<\rho}\log\displaystyle\frac{\rho}{|a_{i}|}\right]+\displaystyle\sum_{|p_{h}|<\rho}\log\displaystyle\frac{\rho}{|p_{h}|}-\displaystyle\sum_{|r_{k}|<\rho}\log\displaystyle\frac{\rho}{|r_{k}|}.
\end{equation*}
Moreover, thanks to Lemma \ref{laplmoebius} and to the linearity of the laplacian, we have that,
\begin{equation*}
\begin{array}{c}
\Delta\log|g (x)|_{|{x=0}}=\Delta\log|f(x)|_{|{x=0}}+\displaystyle\sum_{|r_{k}|<\rho}\displaystyle\frac{2}{\rho^{4}r_{k}^{2}}[r_{k}^{4}-\rho^{4}]-\displaystyle\sum_{|p_{h}|<\rho}\displaystyle\frac{2}{\rho^{4}p_{h}^{2}}[p_{h}^{4}-\rho^{4}]+\\
\\
+\displaystyle\sum_{|a_{i}|<\rho}\displaystyle\frac{2}{\rho^{4}|a_{i}|^{4}}[\rho^{4}-|a_{i}|^{4}][2|a_{i}|^{2}-(a_{i}+ a^{c}_{i})^{2}]-\sum_{|b_{j}|<\rho}\displaystyle\frac{2}{\rho^{4}|b_{j}|^{4}}[\rho^{4}-|b_{j}|^{4}][2|b_{j}|^{2}-(b_{j}+ b^{c}_{j})^{2}].
\end{array}
\end{equation*}

Now, thanks to Theorem \ref{borderbla} for any $x\in\partial \mathbb{B}_{\rho}$ (i.e. for $|x|=\rho$), we have that, for any $i,j,k,h$,
\begin{equation*}
|B_{\mathbb{S}_{b_{j}},\rho}(x)|=|B_{\mathbb{S}_{a_{i}},\rho}(x)|=|B^{p}_{p_{h},\rho}(x)|=|B^{p}_{r_{k},\rho}(x)|=1,
\end{equation*}
and so $\int_{\partial\mathbb{B}_{\rho}}\log |g(y)|d\sigma(y)=\int_{\partial\mathbb{B}_{\rho}}\log |f(y)|d\sigma(y)$.

In the case in which $f$ is a PQL function,
\begin{equation*}
f(x):=a_{0}\prod_{k=1}^{N}(x-q_{k})^{M_{k}}a_{k},
\end{equation*}
the proof goes as before, once we define an appropriate function $g$.
In this case the function $g$ is defined as:
\begin{equation*}
g(x):=\left(\prod_{k=1}^{N}\frac{(\rho^{2}-x q^{c}_{k})^{M_{k}}}{\rho}\right)\prod_{k=N}^{1}(x-q_{k})^{-M_{k}}f(x),
\end{equation*}
of course $\log|g|$ is well defined, meaning that, thanks to the properties of the norm, $|g|$ is not zero or infinite inside the ball.
To prove that this function $g$ is equal to $f$ on $\partial \mathbb{B}_{\rho}$, it is sufficient to adapt the proof of Theorem \ref{borderbla} (and so of Theorem \ref{borderblapun}), in this case
remembering that, for any couple of quaternions $p,q$ it holds $|pq|=|p||q|$.
\end{proof}

As the classical 2D case, our 4D Jensen formulas
relate the mean of a function on the boundary of a ball
centered in zero with radius $\rho$, with the disposition of
its zeros and singularities contained inside the ball.

\begin{remark}\label{mix2}
One key point of the proof are the features of $\rho$-Blaschke factors. These
are built to remove zeros and singularities of the considered functions and to send the set $\partial \mathbb{B}_{\rho}$ onto $\partial\mathbb{B}$. In the following two points of this remark we show how to modify properly these factors to deal with mixed cases.

\begin{enumerate}[label=(\roman*)]
\item In case (ii) of Theorem \ref{jensen}, it is not a problem to assume
some of the quaternions $\{q_{k}\}_{k=1}^{N}$ to lie outside of the ball of radius 
$\rho$. If, in fact, for instance $f(x)=(x-q_{1})(x-q_{2})(x-q_{3})$, with $|q_{1}|,|q_{3}|<\rho$ and $|q_{2}|>\rho$, then, it is sufficient to define an appropriate
function $g$:
\begin{equation*}
g(x):=(x-q_{2})\frac{(\rho^{2}-x q^{c}_{3})}{\rho}\frac{(\rho^{2}-x q^{c}_{1})}{\rho}(x-q_{3})^{-1}(x-q_{2})^{-1}(x-q_{1})^{-1}f(x),
\end{equation*}
and of course, the formula, would involve only the contributes arising from $q_{1}$ and $q_{3}$.
\item A mix of cases (i) and (ii) of the previous theorem can be considered. If, in fact,
we have a quaternionic function $\phi$ defined as $\phi(x)=(x-q_{1})f(x)(x-q_{2})$,
with $f$ a slice preserving semiregular function and, say $|q_{1}|, |q_{2}|<\rho$, it is sufficient to define
$g$ in the following way:
\begin{equation*}
g(x):=\frac{(\rho^{2}-x q^{c}_{2})}{\rho}g_{1}(x) \frac{(\rho^{2}-x q^{c}_{1})}{\rho}(x-q_{2})^{-1}g_{2}(x)(x-q_{1})^{-1}\phi(x),
\end{equation*}
where $g_{1}(x)$ and $g_{2}(x)$ are defined as follows. 
If $\{r_{k}\}_{k=1,2,..}$, $\{p_{h}\}_{h=1,2,..}$ are the sets of real zeros and poles, respectively, and  
$\{\mathbb{S}_{a_{i}}\}_{i=1,2,..}$, $\{\mathbb{S}_{b_{j}}\}_{j=1,2,..}$
are the sets of spherical zeros and poles, respectively, of $f$, everything repeated accordingly to their multiplicity, then:
\begin{equation*}
g_{1}(x):=
\left(\prod_{|p_{h}|<\rho} \frac{\rho^{2}-xp_{h}}{\rho}\prod_{|b_{j}|<\rho}\frac{(x-\rho^{2}b_{j}^{-1})^{s}|b_{j}^{2}|}{\rho^{2}}\right)^{-1}
 \left(\prod_{|r_{k}|<\rho} \frac{\rho^{2}-xr_{k}}{\rho}\prod_{|a_{i}|<\rho}\frac{(x-\rho^{2}a_{i}^{-1})^{s}|a_{i}^{2}|}{\rho^{2}}\right),
\end{equation*}

\begin{equation*}
g_{2}(x)  :=  \left(\prod_{|p_{h}|<\rho} \frac{1}{(x-p_{h})}\prod_{|b_{j}|<\rho}\frac{1}{(x-b_{j})^{s}}\right)^{-1}
 \left(\prod_{|r_{k}|<\rho} \frac{1}{(x-r_{k})}\prod_{|a_{i}|<\rho}\frac{1}{(x-a_{i})^{s}}\right).
\end{equation*}
Observe that the product $g_{1}(x)g_{2}(x)f(x)$
is equal to the function $g$ defined in equation  \eqref{poissonstart}.
\end{enumerate}
\end{remark}

\begin{remark}
In general, given a slice preserving regular function $f$, none of the terms: 
\begin{equation*}
\frac{1}{4}\left(\frac{\rho^{4}-|a_{i}|^{4}}{\rho^{2}|a_{i}|^{4}}\right)[2|a_{i}|^{2}-(a_{i}+ a^{c}_{i})^{2}],
\end{equation*}
in formula \eqref{jensenreg}, can be compensated by one of these:
\begin{equation*}	
\frac{1}{4}\left(\frac{r_{k}^{4}-\rho^{4}}{\rho^{2}r_{k}^{2}}\right).
\end{equation*}
That is, for any $k$ and $i$,
\begin{equation*}
\frac{1}{4}\left(\frac{\rho^{4}-|a_{i}|^{4}}{\rho^{2}|a_{i}|^{4}}\right)[2|a_{i}|^{2}-(a_{i}+ a^{c}_{i})^{2}]+\frac{1}{4}\left(\frac{r_{k}^{4}-\rho^{4}}{\rho^{2}r_{k}^{2}}\right)\neq 0,
\end{equation*}
in fact, the left hand side of the previous equation, is identically zero if and only if,
\begin{equation*}
\frac{1}{4\rho^2}\left(|a_i|^4r_k^4-|a_i|^4\rho^4+\rho^4r_k^2[2|a_{i}|^{2}-(a_{i}+ a^{c}_{i})^{2}]-|a_i|^4r_k^2[2|a_{i}|^{2}-(a_{i}+ a^{c}_{i})^{2}]\right)= 0
\end{equation*}
which entails that the following system holds:
\begin{equation*}
\begin{cases}
r_k^2[2|a_{i}|^{2}-(a_{i}+ a^{c}_{i})^{2}]-|a_i|^4= 0\\
|a_i|^4r_k^2[r_k^2-[2|a_{i}|^{2}-(a_{i}+ a^{c}_{i})^{2}]]= 0.
\end{cases}
\end{equation*}
From the second equation, one gets $r_k^2=[2|a_{i}|^{2}-(a_{i}+ a^{c}_{i})^{2}]$ and plugging this in the first equation, one has $[2|a_{i}|^{2}-(a_{i}+ a^{c}_{i})^{2}]^2=|a_i|^4$.
If now $a_i=\alpha+I\beta$, with $\beta>0$, the last equation is equivalent to
\begin{equation*}
\beta^2(\beta^2+\alpha^2)=0,
\end{equation*}
and this is not possible.

\end{remark}

Another Jensen type formula for the whole class of quaternionic  regular functions was given in \cite{entireslfun}. 
Thanks to Remark \ref{counterexample}, the class of slice preserving semiregular
functions for which is possible to apply the Jensen formula in the form of Theorem
\ref{jensen}, is optimal. However, we do not exclude the possibility to develop other kind of 
formulas relating the value of a generic regular function at a point and its mean
on a sphere containing that point to the disposition of its zeros and poles.
In fact, as already said, the Jensen formula contained in \cite{entireslfun} is 
written for any  regular function but is of course different from the one we just presented.

The main difference is that, while in the one in \cite{entireslfun} the integral
is made on a single slice, in our formula, the integral is over the border
of an Euclidean (4-dimensional) ball. Another difference that comes from this fact
is that, as the reader may observe, in our formula other contributions coming
from the presence of the laplacian of the logarithm of the modulus of Blaschke function appear. 

However, observe that, if one applies our Jensen formula or the Jensen formula contained
in \cite{entireslfun} to a  
$\rho$-Blaschke factor at the sphere $\mathbb{S}_{a}$, $B_{\mathbb{S}_{a},\rho}(x)$,
on the ball $\mathbb{B}_{\rho}$, obtains the same result. This because one has:

\begin{equation*}
\begin{array}{rcl}
\log|B_{\mathbb{S}_{a},\rho}(0)| & = & \displaystyle\frac{1}{|\partial \mathbb{B}_{\rho}|}\displaystyle\int_{\partial \mathbb{B}_{\rho}}\log|B_{\mathbb{S}_{a},\rho}(y)|d\sigma(y)-\displaystyle\frac{\rho^{2}}{8}\Delta\log|B_{\mathbb{S}_{a},\rho}(x)|_{|{x=0}}+\\
& & \\
& & +\left(\log\displaystyle\frac{\rho^{2}}{|a|^{2}}+\displaystyle\frac{1}{4}\left(\displaystyle\frac{\rho^{4}-|a|^{4}}{\rho^{2}|a|^{4}}\right)[2|a|^{2}-(a+ a^{c})^{2}]\right),
\end{array}
\end{equation*}
and now thanks to the fact that $|B_{\mathbb{S}_{a},\rho}(y)|=1$ for any $|y|=\rho$ and thanks to Lemma \ref{laplmoebius}, one has the trivial equality:
\begin{equation*}
\log|B_{\mathbb{S}_{a},\rho}(0)|= \log\displaystyle\frac{\rho^{2}}{|a|^{2}}.
\end{equation*}
The same result can be obtained using the Jensen formula in \cite{entireslfun}.

\subsection{Some corollaries of Jensen formulas}
The last subsection of this work is devoted to state some corollaries of Jensen formulas. Some of them are analogues to results true in complex analysis while other are peculiar of our setting.

The first two results allow us to deal with functions having zeros or singularities in zero or on $\partial \mathbb{B}_{\rho}$. So, in a certain sense, they extend 
our Jensen formulas.

\begin{corollary}\label{genjens1}
Let $f$ be a function that satisfies the hypotheses of Theorem \ref{jensen} but
admits zeros or singularities on $\partial \mathbb{B}_{\rho}$. Then, the two formulas in the
statement still hold in the same form, i.e. no contribution from the zeros and
singularities lying on $\partial \mathbb{B}_{\rho}$ appears.
\end{corollary}
\begin{proof}
First of all, observe that since $\partial \mathbb{B}_{\rho}$ is a compact set, then thanks to Corollaries \ref{zerocptcor} and \ref{polecptcor} the number of zeros or singularities of a semiregular function on it is finite.

 Therefore, we show how to deal in both cases, with a single zero and then the completion of the proof can be simply extended.
Suppose then that $a$ is a (isolated or spherical) zero for $f$ such that $|a|=\rho$, then, consider $r>\rho$, such that $f_{|{\partial \mathbb{B}_{r}}}\neq 0,\infty$. Such $r$ exists because, otherwise there would be too many zeros and poles and either the function is identically zero or not regular at all. If we apply the Jensen formula to $f$ on this new ball $\mathbb{B}_{r}$, then, looking Jensen formulas, the integral,
\begin{equation*}
I(f,r):=\frac{1}{|\partial \mathbb{B}_{r}|}\int_{\partial \mathbb{B}_{r}}\log|f(y)|d\sigma(y),
\end{equation*}
is a continuous function of $r$, therefore the limit
\begin{equation*}
\lim_{r\to\rho}I(f,r)
\end{equation*}
exists and looking again at the Jensen formula, for $r\to \rho$ the terms
involving $a$ vanish. These terms can be of the following forms,
\begin{equation*}
\log\frac{\rho}{|a|},\quad\frac{\rho^{4}-|a|^{4}}{\rho^{2}|a|^{4}}[2|a|^{2}-(a+ a^{c})^{2}]
\quad\mbox{or}\quad \frac{1}{4}\displaystyle\frac{|a|^{4}-\rho^{4}}{\rho^{2}|a|^{2}},
\end{equation*}
therefore, we get the thesis.
\end{proof}

\begin{corollary}\label{genjens2}

Let $f$ be a function that satisfies the hypotheses of Theorem \ref{jensen} but
admits a zero or a singularity at zero, i.e. there exists $k\in\mathbb{Z}\setminus\{0\}$, such that 
\begin{equation*}
f(x)=x^{k}f_{1}(x),
\end{equation*}
with $f_{1}$ that satisfies all the hypotheses of Theorem \ref{jensen}. Then,
Jensen formulas as in Theorem \ref{jensen} (and subsequent remarks), hold, with left hand side equal to,
\begin{equation*}
k\log\rho+\log|f_{1}(0)|.
\end{equation*}

\end{corollary}

\begin{proof}
Suppose that $f(x)=x^{k}f_{1}(x)$, for some $k\in\mathbb{Z}\setminus\{0\}$ and
that $f_{1}$ has nowhere else zeros or poles. Then defining, 
\begin{equation*}
g(x):=\left(\rho x^{-1}\right)^{k}f(x),
\end{equation*}
we have that $g$ satisfies the hypotheses of Theorem \ref{jensen} and that
\begin{equation*}
\log|g(x)|=k\log\rho+\log|f_{1}(x)|.
\end{equation*}
Therefore
\begin{equation*}
\log|g(0)|=\frac{1}{|\partial \mathbb{B}_{\rho}|}\int_{\partial \mathbb{B}_{\rho}}\log |g(y)|d\sigma(y)-\frac{\rho^{2}}{8}\Delta\log |g(x)|_{|x=0},
\end{equation*}
is equivalent to
\begin{equation*}
k\log\rho+\log|f_{1}(0)|=\frac{1}{|\partial \mathbb{B}_{\rho}|}\int_{\partial \mathbb{B}_{\rho}}\log |f(y)|d\sigma(y)-\frac{\rho^{2}}{8}\Delta\log |f(x)|_{|x=0},
\end{equation*}
because, for $y\in\partial \mathbb{B}_{\rho}$, it holds, $|g(y)|=|f(y)|$.
If now $f_{1}$ has zeros or singularities (as in the hypotheses of Theorem \ref{jensen}), then it is sufficient to adapt the former proof changing 
properly the function $g$ in such a way that it gets rid of these zeros and singularities.
The new $g$ will then be as in the proof of Theorem \ref{jensen}, times the new factor  $\left(\rho x^{-1}\right)^{k}$ introduced in the present proof.
\end{proof}

%

We now pass to another group of corollaries that deal with properties of zeros of regular functions.
First of all, let $C$ denotes the following cone:
\begin{equation}\label{cone}
C:=\{\alpha+I\beta\in\mathbb{H}\,|\,\beta\geq|\alpha|,\,\beta\neq0, I\in\mathbb{S}\},
\end{equation}
This cone has an interesting role as the following corollary shows.

\begin{corollary}
Let $f:\Omega\rightarrow\hat{\mathbb{H}}$ be a slice regular function that satisfies the hypotheses of Theorem \ref{jensen}. Then, if all zeros and singularities of $f$ lie in $\partial C$, then the Jensen formula becomes:
\begin{equation*}
\log|f(0)| =  \frac{1}{|\partial \mathbb{B}_{\rho}|}\int_{\partial \mathbb{B}_{\rho}}\log|f(y)|d\sigma(y)-\frac{\rho^{2}}{8}\Delta\log|f(0)| -\sum_{|a_{i}|<\rho}\log\displaystyle\frac{\rho^{2}}{|a_{i}|^{2}}.
\end{equation*}
\end{corollary}
\begin{proof}
The terms in the Jensen formula of the form
\begin{equation*}
\frac{1}{4}\left(\frac{\rho^{4}-|a_{i}|^{4}}{\rho^{2}|a_{i}|^{4}}\right)[2|a_{i}|^{2}-(a_{i}+ a^{c}_{i})^{2}],
\end{equation*}
can be either greater or less or equal to zero, depending on the factor $[2|a_i|^2-(a_i+ a^{c}_i)^2]$. 
If $a_i=\alpha+I\beta$ with $\beta>0$, then, $[2|a_i|^2-(a_i+ a^{c}_i)^2]=\beta^2-\alpha^2$  and so it is greater then zero if and only
if $\beta>|\alpha|$, is equal to zero if and only if $\beta=|\alpha|$ and less than zero otherwise.
Therefore, if
 $f$ is a slice preserving regular function, such that its zeros are all spherical and belonging to   
the boundary $\partial C$, then its Jensen formula becomes:

\begin{equation*}
\log|f(0)| =  \frac{1}{|\partial \mathbb{B}_{\rho}|}\int_{\partial \mathbb{B}_{\rho}}\log|f(y)|d\sigma(y)-\frac{\rho^{2}}{8}\Delta\log|f(x)|_{|x=0} -\sum_{|a_{i}|<\rho}\log\displaystyle\frac{\rho^{2}}{|a_{i}|^{2}}.
\end{equation*}
\end{proof}

\begin{definition}
Let $f$ be a slice regular function defined on the ball centered in zero with radius $R$ and let 
$r<R$. We set the following notations:
\begin{equation*}
M(r)=M_f(r)=\sup_{|x|=r}|f(x)|,\qquad
N(r)=N_f(r)=\sum_{x\in f^{-1}(0)\cap \mathbb{B}_{r}}m_{f}(x).
\end{equation*}
\end{definition}

The next two corollaries give information, under some technical hypotheses,
on the ampleness of the ball centered in zero where a regular function is not zero.

\begin{corollary}\label{corN}
Let $f$ be a slice preserving regular function in a neighborhood of the closed ball $\overline{\mathbb{B}_{R}}$ such that $f(0)\neq 0$
and such that any zero of $f$ lies in the set $C$ defined in formula \eqref{cone}.
Then, if $r<R$, the following inequality holds:
\begin{equation}\label{ineqzero}
N(r)\leq\frac{\log M(R)-\log|f(0)|-\frac{1}{8}R^2\Delta\log|f(x)|_{|x=0}}{\log R-\log r}
\end{equation}
\end{corollary}
\begin{proof}
First of all, since by hypothesis all the zeros of $f$ belong to $C$, then these are spheres passing through some quaternion $a_{i}$. Therefore, if $\delta\in(r,R)$, and if $|a_{1}|\leq |a_{2}|\leq |a_{3}|\leq\dots|a_{N(\delta)}|< \delta$, then by Jensen formula \eqref{jensenreg} and
since $\int_{\mathbb{B}_{\delta}}\log|f(x)|dx\leq |\mathbb{B}_{\delta}|\sup_{\mathbb{B}_{\delta}}\log|f(x)|$,

\begin{equation*}
\begin{array}{rcl}
M(\delta) & \geq & \exp\left(\displaystyle\frac{1}{|\partial \mathbb{B}_{\delta}|}\displaystyle\int_{\partial \mathbb{B}_{\delta}}\log|f(y)|d\sigma(y)\right)=\\
& &\\
 & = & |f(0)|\exp\left(\displaystyle\frac{\delta^{2}}{8}\Delta\log|f(x)|_{|x=0}\right)\displaystyle\prod_{i=1}^{N(\delta)}\left(\displaystyle\frac{\delta^{2}}{|a_{i}|^{2}}\exp\left(\displaystyle\frac{1}{4}\displaystyle\frac{\delta^{4}-|a_{i}|^{4}}{\delta^{2}|a_{i}|^{2}}[2|a_{i}|^{2}-(a_{i}+ a^{c}_{i})^{2}]\right)\right)\geq\\
& &\\
&\geq & |f(0)|\exp\left(\displaystyle\frac{\delta^{2}}{8}\Delta\log|f(x)|_{|x=0}\right)\displaystyle\prod_{i=1}^{N(\delta)}\displaystyle\frac{\delta^{2}}{|a_{i}|^{2}}\geq |f(0)|\exp\left(\displaystyle\frac{\delta^{2}}{8}\Delta\log|f(x)|_{|x=0}\right)\displaystyle\prod_{i=1}^{N(r)}\displaystyle\frac{\delta^{2}}{|a_{i}|^{2}},
\end{array}
\end{equation*}
where the penultimate inequality holds since, for $a_{i}\in C$, the term
\begin{equation*}
\frac{\delta^{4}-|a_{i}|^{4}}{\delta^{2}|a_{i}|^{2}}[2|a_{i}|^{2}-(a_{i}+ a^{c}_{i})^{2}]\geq 0,
\end{equation*}
and so 
\begin{equation*}
\exp\left(\frac{\delta^{4}-|a_{i}|^{4}}{\delta^{2}|a_{i}|^{2}}[2|a_{i}|^{2}-(a_{i}+ a^{c}_{i})^{2}]\right)\geq 1
\end{equation*}
Hence, letting $\delta$ goes to $R$,
\begin{equation*}
\prod_{i=1}^{N(r)}\frac{R^{2}}{|a_{i}|^{2}}\leq \frac{M(R)}{|f(0)|\exp\left(\frac{R^{2}}{8}\Delta\log|f(x)|_{|x=0}\right)}.
\end{equation*}
Now, since $|a_{i}|\leq r$, we have that,
\begin{equation*}
\prod_{i=1}^{N(r)}\frac{R^{2}}{|a_{i}|^{2}}\geq \prod_{i=1}^{N(r)}\frac{R^{2}}{r^{2}}=\left(\frac{R^{2}}{r^{2}}\right)^{N(r)}.
\end{equation*}
We have obtained that
\begin{equation*}
\left(\frac{R^{2}}{r^{2}}\right)^{N(r)}\leq\frac{M(R)}{|f(0)|\exp\left(\frac{R^{2}}{8}\Delta\log|f(x)|_{|x=0}\right)}.
\end{equation*}
Applying now the logarithm to both sides of the last inequality, we get,
\begin{equation*}
\log\left(\frac{R^{2}}{r^{2}}\right)^{N(r)}\leq \log M(R)-\log|f(0)|-\frac{R^{2}}{8}\Delta\log|f(x)|_{|x=0}.
\end{equation*}
Finally, using standard logarithm properties, we obtain the thesis:
\begin{equation*}
N(r)\leq\frac{1}{2}\left(\frac{\log M(R)-\log|f(0)|-\frac{R^{2}}{8}\Delta\log|f(x)|_{|x=0}}{\log R-\log r}\right).
\end{equation*}

\end{proof}

\begin{remark}
If in Corollary \ref{corN}, we set $R=er$, where $e$ is the Napier number, then,
\begin{equation*}
N(r)\leq\frac{1}{2}\left(\log M(er)-\log|f(0)|-\frac{(er)^{2}}{8}\Delta\log|f(x)|_{|x=0}\right).
\end{equation*}
\end{remark}

\begin{corollary}\label{corNN}
Let $f:\mathbb{B}\rightarrow \mathbb{B}$ be a  regular function (not necessarily slice preserving) in a neighborhood of $\overline{\mathbb{B}}$ such that $f(0)\neq 0$
and such that any zero of $f$ lies in the set $C$ defined in equation \eqref{cone}.
If $(|f^{s}(0)|\exp(\frac{1}{8}\Delta\log|f^{s}(x)|_{|x=0}))\leq1$, then $f$ cannot be zero in $\mathbb{B}_{r}$, where $r$ is such that
\begin{equation*}
r<\sqrt{|f^{s}(0)|}\exp\left(\frac{1}{16}\Delta\log|f^{s}(x)|_{|x=0}\right).
\end{equation*}

\end{corollary}

\begin{proof}
First of all, observe that if $q_{0}$ is in $\mathcal{Z}(f)$ and lies in $C$, then
$\mathbb{S}_{q_{0}}$ is a sphere of zeros for $f^{s}$ which lies in $C$ as well.
Now, starting from equation \eqref{ineqzero}, letting $R$ goes to $1$ and remembering that $\log M(R)\leq \log 1=0$, we get,
%
\begin{equation*}
N(r)\leq \frac{1}{2}\left(\frac{\log|f^{s}(0)|+\frac{1}{8}\Delta\log|f^{s}(x)|_{|x=0}}{\log r}\right).
\end{equation*}
Imposing that the right hand side of the last inequality is strictly less than one and since $r<1$, we obtain,
\begin{equation*}
\log r<\frac{1}{2}\left(\log|f^{s}(0)|+\frac{1}{8}\Delta\log|f^{s}(x)|_{|x=0}\right),
\end{equation*}
therefore, if $(|f^{s}(0)|\exp(\frac{1}{8}\Delta\log|f^{s}(x)|_{|x=0}))\leq1$, and passing to the exponential, we obtain the thesis:
\begin{equation*}
r<\sqrt{|f^{s}(0)|}\exp\left(\frac{1}{16}\Delta\log|f^{s}(x)|_{|x=0}\right).
\end{equation*}
\end{proof}

In the previous corollary, if $(|f^{s}(0)|\exp(\frac{1}{8}\Delta\log|f^{s}(x)|_{|x=0}))> 1$,
then any $r$ satisfies the thesis.

In the next two corollaries we show how Jensen formulas can be used to compute some integrals over 3-spheres.

\begin{corollary}\label{measurePQL}
If $f$ is a PQL function
\begin{equation*}
f(x):=a_{0}\prod_{k=1}^{N}(x-q_{k})^{M_{k}}a_{k},
\end{equation*}
with $M_{k}=\pm 1,$ $|q_{k}|<\rho$ and $a_k \in \mathbb{H} \setminus \{ 0 \}$ for any $k$, then,
the following formula holds

\begin{equation*}
 \frac{1}{|\partial \mathbb{B}_{\rho}|}\int_{\partial \mathbb{B}_{\rho}}\log|f(y)|d\sigma(y)  = \frac{\rho^{2}}{8}\left(\sum_{k=1}^{N}M_{k}\frac{2}{|q_{k}|^{2}}\right)+\sum_{h=0}^{N}\log|a_{h}| +
\sum_{k=1}^{N}M_{k}\left(\log\rho+\frac{1}{4}\frac{|q_{k}|^{4}-\rho^{4}}{\rho^{2}|q_{k}|^{2}}\right)
\end{equation*}

\end{corollary}

\begin{proof}
The thesis follows directly from Jensen formula for PQL functions. In fact, it is only  necessary to compute the two quantities $\log|f(0)|$ and $\Delta\log|f(x)|_{|x=0}$ for the given PQL function, but in our case, since
\begin{equation*}
f(x):=a_{0}\prod_{k=1}^{N}(x-q_{k})^{M_{k}}a_{k},
\end{equation*}
then
\begin{equation*}
\log|f(0)|=\sum_{h=0}^{N}\log|a_{h}|+\sum_{k=1}^{N}M_{k}\log|q_{k}|,
\end{equation*}
and thanks to the fact that, for any $x\neq 0$, $\Delta\log|x|=2/|x|^{2}$, then
\begin{equation*}
\Delta\log|f(x)|_{|x=0}=\sum_{k=1}^{N}M_{k}\frac{2}{|q_{k}|^{2}}.
\end{equation*}

\end{proof}

In the last two corollaries we deal with consequences on $\rho$-Blaschke functions coming from Jensen formulas.

\begin{corollary}\label{measure}
Given any $a\in\mathbb{H}\setminus \{0\}$, and any $\rho>0$, if $r>\max\{|a|,\rho^{2}|a|^{-1}\}$, we get,
\begin{equation*}
\frac{1}{|\partial \mathbb{B}_{r}|}\int_{\partial \mathbb{B}_{r}}\log|B_{a,\rho}^p(y)|d\sigma(y)=\log	\frac{|a|}{\rho}+\frac{1}{4}\left(\frac{\rho^{4}-|a|^{4}}{|a|^{2}r^{2}}\right),
\end{equation*}
\begin{equation*}
\frac{1}{|\partial \mathbb{B}_{r}|}\int_{\partial \mathbb{B}_{r}}\log|B_{\mathbb{S}_{a},\rho}(y)|d\sigma(y)=2\log	\frac{|a|}{\rho}-\frac{1}{4}\left(\frac{\rho^{4}-|a|^{4}}{|a|^{4}r^{2}}\right)(2|a|^{2}-(a+ a^{c})^{2}).
\end{equation*}
\end{corollary}
\begin{proof}
The proof uses only the two formulas contained in Theorem \ref{jensen}.
For the usual $\rho$-Blaschke function $B_{a,\rho}^p$,
since $B_{a,\rho}^p(\rho^{2} (a^{c})^{-1})=0$ and $B_{a,\rho}^p(a)=\infty$, the Jensen formula \eqref{jensenpql} states:
\begin{eqnarray*}
\log\frac{\rho}{|a|}&=&\frac{1}{|\partial \mathbb{B}_{r}|}\int_{\partial \mathbb{B}_{r}}\log|B_{a,\rho}^p(y)|d\sigma(y)-\frac{r^{2}}{8}\frac{2}{\rho^{4}|a|^{2}}(|a|^{4}-\rho^{4})+\\
& & +\left(\log\frac{r}{|a|}+\frac{1}{4}\frac{|a|^{4}-r^{4}}{r^{2}|a|^{2}}\right)-\left(\log\frac{r|a|}{\rho}+\frac{1}{4}\frac{\rho^{8}|a|^{-4}-r^{4}}{r^{2}\rho^{4}|a|^{-2}}\right).
\end{eqnarray*}
Starting from this equality, after straightforward computations, we obtain the
thesis.

For the second function, that is the symmetrized of the regular $\rho$-Blaschke function, we have that $B_{\mathbb{S}_{a},\rho}(\mathbb{S}_{\rho^{2} a^{-1}})=0$ and $B_{\mathbb{S}_{a},\rho}(\mathbb{S}_{a})=\infty$ and again, by Jensen formula \eqref{jensenreg},
\begin{eqnarray*}
2\log\frac{\rho}{|a|}&=&\frac{1}{|\partial \mathbb{B}_{r}|}\int_{\partial \mathbb{B}_{r}}\log|B_{\mathbb{S}_{a},\rho}(y)|d\sigma(y)-\frac{r^{2}}{8}\frac{2}{\rho^{4}|a|^{2}}(\rho^{4}-|a|^{4})(2|a|^{2}-(a+ a^{c})^{2})+\\
& & +\left(\log\frac{r^{2}}{|a|^{2}}+\frac{1}{4}\frac{r^{4}-|a|^{4}}{r^{2}|a|^{4}}(2|a|^{2}-(a+ a^{c})^{2})\right)+\\
& & -\left(\log\frac{r^{2}|a|^{2}}{\rho^{4}}+\frac{1}{4}\frac{r^{4}-\rho^{8}|a|^{-4}}{r^{2}\rho^{8}|a|^{-4}}(2\rho^{4}|a|^{-2}-(\rho^{2}a^{-1}+\rho^{2} (a^{c})^{-1})^{2})\right).
\end{eqnarray*}
In this case, the proof goes on, again, by straightforward computations, having in mind that:
\begin{equation*}
(\rho^{2}a^{-1}+\rho^{2} (a^{c})^{-1})=\rho^{2}\frac{a+ a^{c}}{|a|^{2}}.
\end{equation*}
\end{proof}

\begin{remark}
Notice that Corollary \ref{measure}, also implies that,
\begin{equation*}
\lim_{r\to +\infty}\frac{1}{|\partial \mathbb{B}_{r}|}\int_{\partial \mathbb{B}_{r}}\log|B_{a,\rho}^p(y)|d\sigma(y)=\log\frac{|a|}{\rho},
\end{equation*}
and, of course,
\begin{equation*}
\lim_{r\to +\infty}\frac{1}{|\partial \mathbb{B}_{r}|}\int_{\partial \mathbb{B}_{r}}\log|B_{\mathbb{S}_{a},\rho}(y)|d\sigma(y)=2\log\frac{|a|}{\rho},
\end{equation*}
because the two functions
\begin{equation*}
r\mapsto \frac{\rho^{4}-|a|^{4}}{|a|^{2}r^{2}}, \quad r\mapsto \frac{\rho^{4}-|a|^{4}}{|a|^{4}r^{2}}(2|a|^{2}-(a+ a^{c})^{2})
\end{equation*}
are continuous and tend to zero for $r$ that goes to infinity.
\end{remark}

\section*{Acknowledgements}
The authors were partially supported by GNSAGA of INdAM and by FIRB 2012 \textit{Geometria Differenziale e Teoria Geometrica delle Funzioni}. A. Altavilla was also partially supported by the SIR grant 
\textit{NEWHOLITE - New methods in holomorphic iteration} n. RBSI14CFME. 
C. Bisi was also partially supported by PRIN \textit{Variet\'a reali e complesse:
geometria, topologia e analisi armonica}. An important part of the present work was
made while the first author was a postdoc (assegnista di ricerca) at 
Dipartimento di Ingegneria Industriale e Scienze Matematiche of
Universit\`a Politecnica delle Marche.\\
The two authors warmly thanks the two referees for their accurate reading of the present paper and for their really interesting remarks that have improved the paper a lot. \\
Many thanks also to Prof. Irene Sabadini who has asked us the question about Niven polynomials which we have answered in \ref{NivenP}.

\vskip0.5cm
{\bf Keywords} : \keywords{Slice regular functions, bi-harmonic functions, Jensen's formula, Riesz measure.} \\
\vskip0.2cm
{\bf MSC 2010 Primary} : \subjclass{30G35} \\
\vskip0.05cm
{\bf MSC 2010 Secondary} : \subjclass{32A30, 31A30}

%

%
%

\end{document}